\documentclass{article}
\usepackage{verbatim}
\usepackage{amssymb}
\usepackage{amsmath}
\usepackage{amsthm}
\usepackage[all]{xy}
\usepackage{tikz}
\usetikzlibrary{matrix,arrows,decorations.pathmorphing}
\usepackage{tikz-cd}
\tikzset{equal/.style={
      /tikz/arrows=-,
      /tikz/commutative diagrams/double line}
}

\newcommand{\G}{\Gamma}
\newcommand{\g}{\gamma}
\newcommand{\Z}{\mathbb{Z}}
\newcommand{\with}{\mid}

\newcommand{\sgp}[1]{\langle #1\rangle}

\newcommand{\surjarrow}{\twoheadrightarrow}
\newcommand{\injarrow}{\hookrightarrow}
\newcommand{\abs}[1]{\left\lvert#1\right\rvert}
\newcommand{\isom}{\cong}
\newcommand{\quot}[2]{\raisebox{0.5mm}{\ensuremath{#1}}/\raisebox{-0.5mm}{\ensuremath{#2}}}
\newcommand{\fols}{\Rightarrow}

\newcommand{\R}{\mathbb{R}}
\newcommand{\nsub}{\lhd}
\newcommand{\disjcup}{\sqcup}
\newcommand{\bigdisjcup}{\bigsqcup}
\newcommand{\eps}{\varepsilon}
\newcommand{\comp}{\,\text{\raisebox{1pt}{{\tiny $\circ$}}}\,}
\newcommand{\id}{\textrm{id}}

\newcommand{\skel}[2]{{#1}^{(#2)}}

\theoremstyle{plain}
\newtheorem{thm}{Theorem}[section]
\newtheorem{prop}[thm]{Proposition}
\newtheorem{lemma}[thm]{Lemma}
\newtheorem{conj}[thm]{Conjecture}
\newtheorem{cor}[thm]{Corollary}

\newtheoremstyle{namedstyle}
  {}
  {\topsep}
  {\itshape}
  {}
  {\bfseries}
  {.}
  { }
  {#3}
\theoremstyle{namedstyle}
\newtheorem*{named}{Conjecture}

\theoremstyle{definition}
\newtheorem{defn}[thm]{Definition}
\newtheorem{remark}[thm]{Remark}

\newtheorem*{notation}{Notation}
\newtheorem*{unnumclaim}{Claim}

\newcommand{\abut}{\fols}
\newcommand{\nonetwo}{$n$-$(n+1)$-$(n+2)$}

\DeclareMathOperator{\FP}{FP}
\DeclareMathOperator{\wFP}{wFP}
\DeclareMathOperator{\Fin}{F}

\newcommand{\Hom}{H}
\DeclareMathOperator{\coker}{coker}
\makeatletter
\long\def\tlist@if@empty@nTF #1{%
\expandafter\ifx\expandafter\\\detokenize{#1}\\%
\expandafter\@firstoftwo
\else
\expandafter\@secondoftwo
\fi
}
\newcommand{\ses}[4][]{\tlist@if@empty@nTF{#1}{
  #2\injarrow #3\surjarrow #4
  }{
  #2\injarrow #3\stackrel{#1}{\surjarrow} #4
  }
}
\newcommand\B[2][]{%
\mathrm{B}#2%
\tlist@if@empty@nTF{#1}{}{^{(#1)}}
}
\newcommand\E[2][]{%
\mathrm{E}#2%
\tlist@if@empty@nTF{#1}{}{^{(#1)}}
}

\makeatother
\newcommand{\xtilde}{\tilde}

\title{Subdirect products of groups and the $n$-$(n+1)$-$(n+2)$ Conjecture}
\author{Benno Kuckuck}
\date{May 13, 2013}

\begin{document}
\maketitle
\begin{abstract}
We analyse the subgroup structure of direct products of groups.
Earlier work on this topic has revealed that higher finiteness
properties play a crucial role in determining which groups
appear as subgroups of direct products of free groups or limit groups. 
Here, we seek to relate the finiteness properties of a subgroup
to the way it is embedded in the ambient product.
To this end we formulate a conjecture on finiteness properties of
fibre products of groups. We present different approaches to this conjecture,
proving a general result on finite generation of homology groups of
fibre products and, for certain special cases, results on the stronger
finiteness properties $\Fin_n$ and $\FP_n$.
\end{abstract}


\section{Introduction}

This article is about the subgroup structure of direct products of groups.
The questions we will try to shed light on are usually of the following
form: Given some class of groups whose subgroups are reasonably well-understood,
which groups can occur as subgroups of direct products of groups in this
class? How are these subgroups embedded in the ambient product and how does
the way they are embedded reflect on the structure of the subgroup? 

We will explain in this introduction how higher finiteness properties
of groups are central to resolving these questions. We will then propose
a pair of closely related conjectures concerning these properties. 
The first one gives a precise
description for the higher finiteness properties of fibre products of groups,
the fundamental construction that all subgroups of direct products can be built out of
in a certain sense. The second conjecture --- which can be deduced from the first ---
provides a characterisation of the higher finiteness properties of subgroups
of direct products in terms of an embedding condition. In the main part
of the paper we will present our progress on these conjectures, resolving
certain cases and obtaining partial results on the general statement.
In each case we will examine the consequences for the subgroup structure
of direct products.

\paragraph{Background}
In recent years much work has gone into the development of a structure theory
of subgroups of direct products of free groups or, more generally, surface groups and limit groups.
One of the guiding themes that has emerged from this research is an apparent dichotomy
between the subgroups which satisfy strong higher finiteness conditions and those which do not ---
a contrast first observed by Baumslag and Roseblade \cite{BaumslagRoseblade}.
They showed that although in a direct product of two finitely generated free groups
one can find uncountably many isomorphism classes of subgroups,
the only finitely presented subgroups are virtually
direct products of two free groups.

The ``wildness'' of general subgroups of a direct product of free groups
is further exemplified by a number of unsolvability results: For example, in
a direct product $F\times F$ of two free groups on two generators
there is a finitely generated subgroup $L\leq F\times F$ with
unsolvable membership problem \cite{Mihailova} and
unsolvable conjugacy problem \cite{MillerDecisionProblems}.
Furthermore, the isomorphism problem among finitely generated subgroups
is unsolvable \cite{MillerDecisionProblems},
there are uncountably many subgroups with non-isomorphic
abelianisation \cite{BridsonMiller}, and
there is no algorithm to compute the rank of the
abelianisation of an arbitrary finitely generated subgroup \cite{BridsonMiller}.
  
The other side of Baumslag and Roseblade's dichotomy was 
greatly expanded on in a series of papers
by Baumslag, Bridson, Howie, Miller and Short
\cite{BBMS,BHMSSurfaceGroups,BridsonHowie1,BridsonHowie2,BHMSLimitGroups,BHMSFinPres}.
In particular, they generalised the result from \cite{BaumslagRoseblade}
to direct products of any (finite) number of free groups, or more
generally, surface groups or limit groups\footnote{The latter are most easily defined
as the finitely generated \emph{fully residually free} groups, i.e.\ those 
finitely generated groups $\G$ such that
for every finite subset $X\subset\G$ there is a free group
$F$ and a homomorphism $f:\G\to F$ which is injective on $X$. This class
includes free groups, free abelian groups and fundamental groups of
orientable surfaces.}.
This generalisation says
that a subgroup $P\leq \G_1\times\dots\times\G_n$ in a direct product
of finitely generated free (surface, limit) groups $\G_1,\dots,\G_n$, which is of type $\Fin_n$, is itself
virtually a direct product of $n$ or fewer finitely generated free (surface, limit) groups \cite[Theorem A]{BHMSSurfaceGroups}, 
\cite[Theorem A]{BHMSLimitGroups}.
In fact, if the intersection of $P$ with each factor $\G_i$ is non-trivial,
then the product of these intersections will have finite index in
$P$ \cite[Theorem B]{BHMSSurfaceGroups}, \cite[Theorem C]{BHMSLimitGroups}. 
So, virtually, the only subgroups of type $\Fin_n$ are the ``obvious'' ones.

In both results the higher finiteness condition is essential:
The well-known examples due to Stallings \cite{Stallings} and
Bieri \cite{BieriQMN} of groups of type $\Fin_{n-1}$ which are not of type $\Fin_n$
arise as subgroups of direct products of free groups and serve to show
that the finiteness condition in the theorems by Bridson, Howie, Miller and Short
cannot be replaced by mere finite presentability.

This indicates that a thorough study of the
higher finiteness properties of subgroups of direct products
will be necessary in order to gain a full understanding of this class of groups.
In particular, there is great interest in
comprehending the class of finitely presented subgroups
of direct products of limit groups, as these are just the finitely
presented residually free groups. A systematic study of these
groups was started in \cite{BHMSFinPres}.

In light of the work cited above, one might think of
the subgroups of a direct product of $n$ free (or limit) groups
as arranged along a spectrum according to which finiteness properties
they satisfy. At the lower end, the class of all finitely generated
subgroups is hopelessly complicated. At the upper end, the only
subgroups of type $\Fin_n$ are the obvious ones.
But if we aim to understand all the finitely presented groups
we need to find out what happens as one imposes successively stronger
finiteness conditions. The hope is that the stark contrast
between ``finitely generated'' and ``type $\Fin_n$'' will,
to some extent, be resolved by considering the classes in between.

\paragraph{The Main Conjectures}
Our approach is built around two main conjectures. The first gives a criterion
for a subgroup of a direct product to be of type $\Fin_k$, in terms
of how it is embedded in the ambient product:
\begin{named}[The Virtual Surjections Conjecture]
  Let $\G_1,\dots,\G_n$ be groups of type $\Fin_k$ and
  $P\leq\G_1\times\dots\times\G_n$ a subgroup of their
  direct product. Assume that $P$ \emph{virtually surjects to $k$-tuples of factors},
  i.e.\ for any $i_1,\dots,i_k$ with $1\leq i_1<\dots<i_k\leq n$ the image of
  $P$ under the projection $P\to\G_{i_1}\times\dots\times\G_{i_k}$ is
  of finite index. Then $P$ is of type $\Fin_k$.
\end{named}

This conjecture appears in the literature in various forms,
e.g.\ \cite[Section 5]{Kochloukova} 
or \cite[Section 12.5]{DisonThesis}.

The case $k=2$ of this conjecture has been proved by Bridson, Howie, Miller
and Short in \cite{BHMSFinPres}. 
While the conjecture makes no assumptions on the factor groups, beyond
the finiteness condition $\Fin_k$, it is especially pertinent
for the case where all the $\G_i$ are free groups (or limit groups).
In this case, with certain non-triviality assumptions,
the converse holds, 
as was shown by Kochloukova \cite{Kochloukova}. So resolving
the Virtual Surjections Conjecture in this case would give a complete characterisation
of the higher finiteness properties of subgroups of direct products of
limit groups.

Any subgroup
of a direct product can be built up as an iterated fibre product.
We use this fact to make progress on the Virtual Surjections Conjecture
by studying fibre products.
In particular we investigate another
conjecture which proposes a precise
set of conditions for a fibre product to be of type
$\Fin_n$.

\begin{named}[The $n$-$(n+1)$-$(n+2)$ Conjecture]
  Let
  \begin{gather*}
  \ses[\pi_1]{N_1}{\Gamma_1}Q\\
  \ses[\pi_2]{N_2}{\Gamma_2}Q
  \end{gather*}
  be two short exact sequences of groups and
  \[P:=\{(\gamma_1,\gamma_2)\in\Gamma_1\times\Gamma_2\with \pi_1(\gamma_1)=\pi_2(\gamma_2)\}\]
  their \emph{fibre product}.
  
  If $N_1$ is of type $\Fin_n$, both $\Gamma_1$ and $\Gamma_2$ are of
  type $\Fin_{n+1}$ and $Q$ is of type $\Fin_{n+2}$, then
  $P$ is of type $\Fin_{n+1}$.
\end{named}

The $n=1$ case of this is the $1$-$2$-$3$ Theorem, proved in \cite{BBMS} and
\cite{BHMSFinPres2}.

\paragraph{Results}
Our first result is the above-mentioned implication:
\begin{named}[Theorem \ref{thm:implication}]
  If the $n$-$(n+1)$-$(n+2)$ Conjecture holds whenever $Q$ is nilpotent,
  then the Virtual Surjections Conjecture holds.
\end{named}

In Section \ref{sec:top}, we prove a special case of the $n$-$(n+1)$-$(n+2)$ Conjecture:
\begin{named}[Corollary \ref{cor:splitn12}]
The $n$-$(n+1)$-$(n+2)$ Conjecture holds whenever the second sequence splits.
\end{named}

We also prove the following reduction:
\begin{named}[Proposition \ref{prop:reductiontofree}]  
  If the $n$-$(n+1)$-$(n+2)$ Conjecture holds whenever $\G_2$
  is finitely generated free then it holds in general.
\end{named}

In Section \ref{sec:weak} we prove a homological
variant of the $n$-$(n+1)$-$(n+2)$ Conjecture. We say that a group $\G$
is of type $\wFP_n$ (``weak $\FP_n$'') if the homology
groups $\Hom_k(\G,\Z)$ are finitely generated for $k\leq n$. This condition
has appeared in the above-mentioned work by Bridson, Howie, Miller and Short
and has proved useful there (in particular, in many of the theorems
we quoted earlier, the $\Fin_n$-conditions may be replaced by corresponding
$\wFP_n$-conditions).
\begin{named}[Theorem \ref{thm:weakn12} (The Weak $n$-$(n+1)$-$(n+2)$ Theorem)]
  Let
  \begin{gather*}
    \ses[\pi_1]{N_1}{\G_1}Q\\
    \ses[\pi_2]{N_2}{\G_2}Q
  \end{gather*}
  be short exact sequences with $N_1$ of type $\wFP_n$,
  $\G_1$ of type $\wFP_{n+1}$, $\G_2$ of type $\FP_{n+1}$,
  and $Q$ of type $\FP_{n+2}$. Then the associated fibre product
  is of type $\wFP_{n+1}$.
\end{named}

This implies a corresponding variant of the Virtual Surjections Conjecture:
\begin{named}[Corollary \ref{cor:weakvsc} (The Weak Virtual Surjections Theorem)]
  Let $k\geq 2$. Let $\G_1,\dots,\G_n$ be groups of type $\FP_k$ and $P\leq\G_1\times\dots\times\G_n$
  a subgroup of their direct product. If $P$ virtually surjects
  to $k$-tuples of factors then $P$ is of type $\wFP_k$.
\end{named}

This, in conjunction with the converse for the case of direct products of
limit groups, proved by Kochloukova \cite[Theorem 11]{Kochloukova} yields:
\begin{named}[Corollary \ref{cor:wFPkCharacterisation}]
  Let $\G_1,\dots,\G_n$ be non-abelian limit groups
  and $P\leq\G_1\times\dots\times\G_n$ a subdirect product
  such that $P\cap\G_i\neq 1$ for all $i$.
  Then $P$ is of type $\wFP_k$ (for some $k\geq 2$) if and only if $P$ virtually surjects to $k$-tuples.
\end{named}

Finally, in Section \ref{sec:abelian}, we treat another special case
of the two main conjectures:
\begin{named}[Theorem \ref{thm:abeliann12}]
  The $n$-$(n+1)$-$(n+2)$ Conjecture holds whenever $Q$ is virtually abelian.
\end{named}

As a consequence, we again obtain a corresponding case of the Virtual Surjections
Conjecture. If $P\leq\G_1\times\dots\times\G_n$ is a subdirect product\footnote{i.e.\ all
the projections $P\to\G_i$ are onto}, following \cite{BBMS}, we say
it is of Stallings-Bieri type if the $\G_i/(P\cap\G_i)$ are virtually abelian.
Then $P$ is virtually the kernel of a homomorphism to an abelian group
(see Corollary \ref{cor:virtabelian} for a precise statement).
Recovering a result by Kochloukova, we show that the Virtual Surjections
Conjecture holds for subgroups of Stallings-Bieri type.

Furthermore, we show that any subgroup $P\leq\G_1\times\dots\times\G_n$
of a direct product which virtually
surjects to $k$-tuples for some $k>\frac{n}{2}$
is of Stallings-Bieri type (Corollary \ref{cor:virtabelian}).

We thus obtain:
\begin{named}[Corollary \ref{cor:fewfactorsvsc}]
  The Virtual Surjections Conjecture holds for $k>\frac{n}{2}$.
\end{named}

\paragraph{Funding} This work was supported by a grant from the Engineering
and Physical Sciences Research Council.

\paragraph{Acknowledgements} I would like to
thank my DPhil supervisor, Martin Bridson, for introducing me to the topic
and for his invaluable help and guidance in preparing this paper.

\section{Fibre products and subdirect products}
\label{sec:fibreproducts}

We shall start off by giving the basic definitions and
recalling some fundamental facts about
the construction of fibre products and how that basic construction can be used
to iteratively
build up subdirect products of groups.

\begin{defn}
A subdirect product of the groups $\Gamma_1,\dots,\Gamma_n$ is
a subgroup $P\leq\Gamma_1\times\dots\times\Gamma_n$ of the direct product
such that the canonical projections $P\to\Gamma_i$ are surjective
for all $i$.
\end{defn}

When studying subgroups of direct products it is often possible to
restrict to the case of subdirect products: If $P\leq\G_1\times\dots\times\G_n$
is a subgroup with the projection homomorphisms denoted by
$p_i:\G_1\times\dots\times\G_n\to\G_i$, then $P$ is a subdirect product
of $p_1(P),\dots,p_n(P)$.

One way of obtaining a subdirect product of two groups
is via the fibre product construction:
Let
\begin{gather*}
\ses[\pi_1]{N_1}{\G_1}Q\\
\ses[\pi_2]{N_2}{\G_2}Q
\end{gather*}
be two short exact sequences. Then we call
\[P:=\{(\g_1,\g_2)\in\G_1\times\G_2\with\pi_1(\g_1)=\pi_2(\g_2)\}\]
the fibre product associated to these sequences (or to $\pi_1$ and $\pi_2$).\footnote{If the two 
sequences are identical, i.e.\ $\pi_1=\pi_2$, $P$ 
is often called a ``symmetrical'' or ``untwisted'' fibre product (associated to $\pi_1$), but
we will make little use of this concept.}

A crucial observation is that to any such fibre product we can
associate two exact sequences of the form $\ses{N_1}P{\G_2}$
and $\ses{N_2}P{\G_1}$. More precisely:
\begin{lemma}
  The fibre product above fits into a commutative diagram
  \[\begin{tikzcd}
  &1\times N_2\ar[equal]{r}\ar[hook]{d}&N_2\ar[hook]{d}\\
  N_1\times 1\ar[hook]{r}\ar[equal]{d}&P\ar[two heads]{r}{p_2}\ar[two heads]{d}{p_1}&\G_2\ar[two heads]{d}{\pi_2}\\
  N_1\ar[hook]{r}&\G_1\ar[two heads]{r}{\pi_1}&Q
  \end{tikzcd}\]
  with exact rows and columns, where $p_1:P\to\G_1$ and $p_2:P\to\G_2$ are
  the restrictions of the canonical projections $\G_1\times\G_2\to\G_i$
  and all the injections are inclusion homomorphisms.
\end{lemma}
\begin{proof}
  The lower right square commutes simply by the definition
  of the fibre product. If $(\g_1,\g_2)\in P$ is in the kernel
  of $p_2$, then $\g_2=1$ and $\pi_1(\g_1)=\pi_2(\g_2)=1$, 
  so $\g_1\in\ker\pi_1=N_1$. Obviously, $N_1\times 1$ is sent 
  isomorphically to $N_1\leq\G_1$ under the projection $p_1$.
  The rest follows by symmetry.
\end{proof}

We will usually identify $N_1\times 1$ with $N_1$ and
$1\times N_2$ with $N_2$.

It is a basic observation that every subdirect product of
two groups is a fibre product. Precisely, we have the following:

\begin{lemma}
  Let $P\leq\Gamma_1\times\Gamma_2$ be a subdirect product\label{subdfibre}
  of $\Gamma_1$ and $\Gamma_2$. Let $N_1:=P\cap\Gamma_1$ and $N_2:=P\cap\Gamma_2$, 
  where $\Gamma_1$ denotes
  the subgroup $\Gamma_1\times 1\leq\Gamma_1\times\Gamma_2$ and similarly
  for $\Gamma_2$ (i.e.\ we have identified each $\Gamma_i$ with its image
  under the natural inclusion $\Gamma_i\injarrow\Gamma_1\times\Gamma_2$). Then
  \[\quot{\G_1}{N_1}\isom \quot{P}{N_1\times N_2}\isom\quot{\G_2}{N_2}\]
  and $P$
  is the fibre product associated to two short exact sequences
  \begin{alignat*}{3}
  \ses{N_1}{&\,\G_1}{&\,Q}\\
  \ses{N_2}{&\,\G_2}{&\,Q}&.
  \end{alignat*}
\end{lemma}
\begin{proof}
  Let $p_i:P\to\G_i$ for $i=1,2$ denote the restriction of the canonical projection
  $\G_1\times\G_2\to\G_i$. Clearly, $N_1$ is the kernel of $p_2$, so
  $N_1$ is normal in $P$. Symmetrically, $N_2$ is normal in $P$, as well.
  Therefore $N_1\times N_2$ is normal in $P$ and the quotient map
  $P\to P/(N_1\times N_2)$ factors over both projections $p_i$, giving
  a commutative square
  \[\begin{tikzcd}
    P\ar{r}{p_2}\ar{d}[swap]{p_1}&\G_2\ar{d}{\pi_2}\\
    \G_1\ar{r}{\pi_1}&\quot{P}{N_1\times N_2}
  \end{tikzcd}\]
  The kernel of $\pi_1$ is $p_1(N_1\times N_2)=N_1$, and similarly
  $\ker\pi_2=N_2$. Denoting $P/(N_1\times N_2)$ by $Q$, we thus have
  two short exact sequences:
  \begin{alignat*}{3}
  \ses[\pi_1]{N_1}{&\,\G_1}{&\,Q}\\
  \ses[\pi_1]{N_2}{&\,\G_2}{&\,Q}&.
  \end{alignat*}
  Let $P'$ denote their fibre product. Commutativity of the above square
  means that $P\leq P'$. Conversely, if $\g'=(\g_1',\g_2')\in P'$, set
  $q:=\pi_1(\g_1')=\pi_2(\g_2')$. Then
  pick a $\g:=(\g_1,\g_2)\in P$ with 
  \[\g(N_1\times N_2)=q.\]
  Then $\pi_i(\g_i)=q=\pi_i(\g_i')$ for $i=1,2$ by definition of the maps $\pi_i$.
  And hence $\g\g'^{-1}\in N_1\times N_2\subset P$ so $\g'\in P$.
  
  This shows that $P=P'$, so $P$ is the fibre product associated to the
  two short exact sequences.
\end{proof}

Indeed, not only is every subdirect product of two groups
a fibre product, but we can build up any subdirect product
of three or more groups by iterating the fibre product construction.
The following notation will be frequently useful, when dealing with subgroups
of direct products of more than two groups:

\begin{notation}
Let
$\G_1,\dots,\G_n$ be groups and
$P\leq\Gamma:=\Gamma_1\times\dots\times\Gamma_n$ a subgroup of their
direct product.
For $J=\{i_1,\dots,i_m\}\subset\{1,\dots,n\}$ with $i_1,\dots,i_m$ pairwise distinct,
we will denote by $\Gamma_J$ the group $\Gamma_{i_1}\times\dots\times\Gamma_{i_m}$,
which we will also, via the canonical injection, consider a subgroup of $\Gamma$.
Furthermore, we will denote by $p_J$ the canonical projection map $P\to\Gamma_J$
and by $N_J$ the intersection $P\cap\Gamma_J$. Sometimes we will omit curly braces 
to unclutter
the notation, e.g.\ we will write $p_{1,\dots,n-1}$ for the projection
$P\to\Gamma_1\times\dots\times\Gamma_{n-1}$, in place of $p_{\{1,\dots,n-1\}}$.
\end{notation}

Now let $P\leq\Gamma_1\times\dots\times\Gamma_n$ be a subdirect product. 
Using the notation introduced above, set
\[T:=p_{1,\dots,n-1}(P)\leq\Gamma_1\times\dots\times\Gamma_{n-1}.\]
Then $P$ is a subdirect product of $T$ and $\G_n$, so by Proposition \ref{subdfibre} 
it is the fibre product
associated to short exact sequences
\begin{alignat*}{2}
\ses{N_{1,\dots,n-1}}{&\;T&}Q\\
\ses{N_n}{&\,\G_n\,&}Q&.
\end{alignat*}
Note that $T$, in turn, is a subdirect product of $\Gamma_1,\dots,\Gamma_{n-1}$.
This allows us to argue about subdirect products of any number of groups
by using induction
on the number of factors and dealing with a fibre product at every step.

\section{The $n$-$(n+1)$-$(n+2)$ Conjecture and the Virtual Surjections Conjecture}
\label{sec:vsc}
In the Introduction we explained that one of the major motivations
for studying the $n$-$(n+1)$-$(n+2)$ Conjecture is
that it implies the Virtual Surjections Conjecture, which
relates the finiteness properties of subdirect products
to the way these groups are
embedded in the ambient direct product. The main goal of the present section
is to prove this implication. En passant, we will
observe a curious connection between subdirect 
products and nilpotent groups.

Let us restate the embedding condition from the Introduction, which we can now phrase using the notation
introduced earlier:
\begin{defn}
  Let $\G_1,\dots,\G_n$ be groups and
  let $P\leq\G_1\times\dots\times\G_n$ be a subgroup
  of their direct product. We say that $P$ virtually
  surjects to $k$-tuples of factors if
  $p_J(P)\leq\G_J$ is a finite index subgroup for
  every $k$-element subset of indices $J\subset\{1,\dots,n\}$.
\end{defn}

Recall the conjecture from the Introduction.
\begin{conj}[The Virtual Surjections Conjecture]
  Let $k\geq 2$.\label{conj:vsc}
  Let $\G_1,\dots,\G_n$ be groups of type $\Fin_k$ and
  $P\leq\G_1\times\dots\times\G_n$ a subgroup
  of their direct product.
  If $P$ virtually surjects to $k$-tuples of factors
  then $P$ is of type $\Fin_k$.
\end{conj}

We will now embark on the task of showing that the
$n$-$(n+1)$-$(n+2)$ Conjecture implies the Virtual Surjections Conjecture.

\subsection{Virtually Nilpotent Quotients}
\label{subsec:VirtuallyNilpotent}

As our first ingredient we need an important consequence
of the property of virtually surjecting to $k$-tuples
for some $k\geq 2$.
This lemma originated in the ``Virtually Nilpotent Quotients Theorem'' from
\cite[Section 4.3]{BridsonMiller}, see also \cite[Proposition 3.2]{BHMSFinPres}. 
We will prove a variation with a precise
bound on the nilpotency class, which we will exploit in Section \ref{sec:abelian}.

\begin{lemma}
  Let $\G_1,\dots,\G_n$ be groups and\label{lem:nilp}
  $P\leq\Gamma_1\times\dots\times\Gamma_n$ a subdirect product which
  virtually surjects to $k$-tuples for some $k\geq 2$. Then for every index $i$ 
  the group $\Gamma_i/N_i$
  is virtually nilpotent of class at most $\lceil\frac{n-1}{k-1}\rceil-1$.
  \end{lemma}
\begin{proof}
  We will show the statement only for $i=1$. For other choices of $i$ the proof is 
  entirely analogous (or alternatively, can be reduced to the case treated here, 
  using the isomorphism exchanging the first and $i$\textsuperscript{th} factor of the product).

  Let $s:=\lceil\frac{n-1}{k-1}\rceil$. Then the set $I:=\{2,\dots,n\}$ can be partitioned into $s$ disjoint
  subsets $I_1,\dots,I_s\subset I$ of size $\abs{I_m}\leq k-1$ for all $m$. Set $I_m':=\{1\}\cup I_m$.
  Since $P$ virtually surjects to $k$-tuples, $p_{I_m'}(P)$ is of finite index in $\Gamma_{I_m'}$ for
  all $m$. Therefore $\Gamma_1':=\bigcap_{m=1}^s (p_{I_m'}(P)\cap\Gamma_1)$ is of finite index
  in $\Gamma_1$. Note that $\Gamma_1'$ consists of all those $\gamma\in\Gamma_1$ such that 
  for each $m$ with $1\leq m\leq s$ there exists a $g\in P$ with
  $p_1(g)=\gamma$ and $p_i(g)=1$ for all $i\in I_m$. In particular $\Gamma_1'$ contains
  $N_1$.

  Now let $\gamma_1,\dots,\gamma_s\in\Gamma_1'$. We will show that the iterated commutator
  \[c:=[\gamma_1,[\gamma_2,\dots,[\gamma_{s-1},\gamma_s]\dots]]\]
  is in $N_1$.
  Pick elements $g_1,\dots,g_s\in P$ such that $p_1(g_m)=\gamma_m$ for $1\leq m\leq s$ 
  and $p_i(g_m)=1$ for $i\in I_m$.
  Now let
  \[c':=[g_1,[g_2,\dots,[g_{s-1},g_s]\dots]].\]
  Then $c'$ is in $P$ and $p_1(c')=c$. Furthermore, for each $i$ with $2\leq i\leq n$ there
  is an $m$ with $1\leq m\leq s$ such that $i\in I_m$, so $p_i(g_m)=1$. But clearly, if one of the elements in an iterated commutator
  equals $1$, then the commutator itself equals $1$, too. So for each $i$ with $2\leq i\leq n$, we get that
  $p_i(c')=1$. But this means that $c=c'\in\Gamma_1\cap P=N_1$.

  We have shown that $N_1$ contains all the commutators of length $s$ in $\Gamma_1'$
  and therefore it also contains $\gamma_{s-1}\Gamma_1'$, the $(s-1)$\textsuperscript{th}
  term in the lower central series of $\Gamma_1'$. This proves that $\Gamma_1'/N_1$
  is nilpotent of class at most $s-1$.
\end{proof}

In particular we record the following generalisation of
\cite[Lemma 4.9]{BridsonMiller}.

\begin{cor}
  Let $\G_1,\dots,\G_n$ be groups and let
  $P\leq \G_1\times\dots\times\G_n$ be a subgroup
  of their direct product which virtually surjects\label{cor:virtabelian}
  to $k$-tuples for some $k>\frac{n}{2}$.
  Then $p_i(P)/N_i$ is virtually abelian for all $i$.
  
  Furthermore, $P$ is virtually the kernel of a homomorphism to
  an abelian group. Precisely: There are finite index subgroups
  $\G_i^0\leq\G_i$ for each $i$ and a homomorphism
  $\phi:\G_1^0\times\dots\times\G_n^0\to A$ to an abelian group $A$,
  such that $P\cap(\G_1^0\times\dots\times\G_n^0)=\ker\phi$.
\end{cor}
\begin{proof}
  Let $\G_i':=p_i(P)\leq\G_i$. These are finite index subgroups,
  since $P$ virtually surjects to $k$-tuples. Then $P$ is a subdirect product
  in $\G_1'\times\dots\times\G_n'$, which still virtually surjects to $k$-tuples.
  Now, since $k>\frac n2$, we have $k\geq \frac n2+\frac 12$ and so
  $\frac{n-1}{k-1}\leq 2$. So Lemma \ref{lem:nilp} says that 
  $\G_i'/(P\cap\G_i')=\G_i'/N_i$ is virtually abelian for all $i$.
  
  Hence there are finite index subgroups $\G_i^0\leq \G_i'$, containing $N_i$,
  for all $i$, such that $\G_i^0/N_i$ is abelian. Then $N_i$ contains
  the commutator subgroup $[\G_i^0,\G_i^0]$. Therefore
  $N_1\times\dots\times N_n$ contains the commutator subgroup
  of $\G_1^0\times\dots\times\G_n^0$. But then 
  $P\cap(\G_1^0\times\dots\times\G_n^0)$ contains
  this commutator subgroup as well, and in particular it is
  normal in $\G_1^0\times\dots\times\G_n^0$ with
  abelian quotient group.
\end{proof}

Corollary \ref{cor:virtabelian} is particularly interesting
in light of the following theorem by Koch\-lou\-kova \cite[Theorem 11]{Kochloukova}:
\begin{thm}
  If $\G_1,\dots,\G_n$ are non-abelian free (or limit) groups
  and $P\leq\G_1\times\dots\times\G_n$ is a full subdirect product
  of type $\Fin_k$ then $P$ virtually surjects to $k$-tuples.\qed\label{thm:KochloukovasTheorem}
\end{thm}

We can thus conclude:
\begin{cor}
  If $\G_1,\dots,\G_n$ are non-abelian free (or limit) groups
  and $P\leq\G_1\times\dots\times\G_n$ is a full subdirect product
  of type $\Fin_k$ with $k>\frac{n}{2}$ then $P$ is virtually
  the kernel of a map to an abelian group (in the same sense
  as in the statement of Corollary \ref{cor:virtabelian}).\qed
\end{cor}

The subgroups which are virtually kernels of maps to abelian 
groups are called ``of Stallings-Bieri type'' in \cite{BridsonMiller},
alluding to the examples in \cite{Stallings} and \cite{BieriQMN}
of groups of type $\Fin_{n-1}$ with non-finitely generated
$n$\textsuperscript{th} integral homology group, which are
kernels of maps from a direct product of $n$ finitely generated
free groups to $\Z$. For a while,
it was an open question whether every finitely presented
full subdirect product of free groups is of this type
\cite[Question 4.3]{BridsonMiller}, until examples
proving that this is not the case were discovered
in \cite[Theorem H]{BHMSFinPres2} (or see \cite[Theorem H]{BHMSFinPres}).

We will have more to say on these kinds of subgroups
in Chapter \ref{sec:abelian} (in particular we will prove the
Virtual Surjections Conjecture for them), but let us for now get back to
the matter at hand.

\subsection{The $0$-$1$-$2$ Lemma and virtual surjection to factors}

Next, we will investigate the case $k=1$ of the Virtual Surjections Conjecture.
We have excluded this case from the
Conjecture with good reason: It is not true for $k=1$.
To see this, note that Bridson and Miller show in \cite[Corollary 2.4]{BridsonMiller}
that a subdirect product $P\leq\G_1\times\G_2$ 
of two finitely presented groups $\G_1$ and $\G_2$ is finitely
generated if and only if $\G_1/(P\cap \G_1)\isom \G_2/(P\cap\G_2)$
is finitely presented. So, choosing any finitely generated but not finitely
presented group $Q$ and an epimorphism $\pi:\G\surjarrow Q$ where
$\G$ is a finitely generated free group, the symmetric fibre product
associated to $\pi$ will not be finitely generated.

We can, however, make the $k=1$ case of the Virtual
Surjections Conjecture become true, if we add the extra assumption that
the $p_i(P)/N_i$ be virtually nilpotent for all $i$ --- as they are
when $k\geq 2$ by Lemma \ref{lem:nilp}.

To show this, we will need the $0$-$1$-$2$ Conjecture.
This is really no more than an easy lemma and we omit the proof,
which can be found in \cite[Proposition 2.3 and Corollary 2.4]{BridsonMiller}.

\begin{lemma}[The $0$-$1$-$2$ Lemma]
  Let
  \begin{gather*}
    \ses[\pi_1]{N_1}{\G_1}Q\\
    \ses[\pi_2]{N_2}{\G_2}Q
  \end{gather*}
  be two short exact sequences with $\G_1$ and $\G_2$ finitely generated
  and $Q$ finitely presented. Then the associated fibre product
  is finitely generated.\label{lem:012}\qed
\end{lemma}

The ``$k=1$ case'' of the
Virtual Surjections Conjecture is as follows:
\begin{prop}
  Let $\G_1,\dots,\G_n$ be finitely \label{prop:vs1} generated groups and
  $P\leq\G_1\times\dots\times\G_n$ a subgroup of their direct product.
  If the image $p_i(P)\leq\G_i$ under the projection to each factor 
  $\G_i$ is of finite index in $\G_i$ and $p_i(P)/(P\cap\G_i)$ is virtually nilpotent for all $i$
  then $P$ is finitely generated.
\end{prop}
\begin{proof}
  Note first that all the $\G_i':=p_i(P)$ are finitely generated, since
  they are of finite index in $\G_i$, and that
  $p_i(P)/(P\cap\G_i')=p_i(P)/(P\cap\G_i)$.
  Thus, replacing all the $\G_i$ by $\G_i'$, we may assume without loss
  of generality that $P$ is a subdirect product.
  
  We now prove the claim by induction on $n$. If $n=1$, there is nothing
  to prove. So assume $n\geq 2$. Let
  $T:=p_{1,\dots,n-1}(P)$. As described at the end of Chapter \ref{sec:fibreproducts}, $P$ is the fibre product
  associated to the short exact sequences
  \begin{alignat*}{2}
  \ses{N_{1,\dots,n-1}}{&\,T&}Q\\
  \ses{N_n}{&\,\G_n\,&}Q&.
  \end{alignat*}
  Here, $T$ is a subdirect product of
  $\G_1,\dots,\G_{n-1}$. Note that
  $T\cap\G_i$ for $1\leq i\leq n-1$ contains
  $P\cap\G_i$, therefore $\G_i/(T\cap\G_i)$ is a quotient
  of the virtually nilpotent group $\G_i/(P\cap\G_i)$,
  and hence virtually nilpotent itself.
  Thus we can conclude by induction that $T$ is finitely generated.
  By assumption, $Q\isom\G_n/N_n$ is finitely generated and virtually nilpotent
  and therefore finitely presented.
  Hence we can apply the $0$-$1$-$2$ Lemma to conclude that $P$ is finitely generated.
\end{proof}
\begin{remark}
  Note that the nilpotency condition
  on the groups $\G_i/N_i$ is only used in the penultimate sentence of the proof,
  to establish that certain
  finitely generated groups are finitely presented.
  Going through the proof carefully, we can see that the condition
  that $\G_i/N_i$ be virtually nilpotent can be replaced with the much
  weaker condition that $\G_i/(p_J(P)\cap\G_i)$ be finitely
  presented for all $i$ and all $J\subset\{1,\dots,n\}$
  with $i\in J$. In particular, this holds if all quotients
  of $\G_i/N_i$ are finitely presented. For finitely presented
  groups this condition is sometimes called
  Max-$n$, as the property of having only finitely presented quotients
  is equivalent to satisfying the ``maximal condition for normal subgroups''
  (i.e.\ any ascending chain of normal subgroups is eventually constant).
  Other examples of finitely presented Max-$n$ groups --- apart from the
  finitely generated virtually nilpotent groups --- are virtually polycyclic
  groups and finitely presented metabelian groups \cite{BieriStrebel} (though
  there are finitely presented soluble groups which are not Max-$n$ \cite{Abels}).
\end{remark}

\subsection{The $n$-$(n+1)$-$(n+2)$ Conjecture implies the Virtual Surjections Conjecture}
\label{subsec:implies}

Before we finally show the main result of this chapter, we need two quick lemmas.

\begin{lemma}
  Let $P\leq\G_1\times\dots\times\G_n$. Let
  $I,J\subset\{1,\dots,n\}$ with $I\cup J=\{1,\dots,n\}$
  and $I\cap J\neq\emptyset$. Then
  $p_{I\cap J}(P\cap\G_I)=p_J(P)\cap\G_{I\cap J}$.
\end{lemma}
\begin{proof}
  Let $\g\in p_{I\cap J}(P\cap\G_I)$. Pick
  $p\in P\cap\G_I$ with $p_{I\cap J}(p)=\g$.
  Now $p_J(p)\in \G_{I\cap J}$, therefore
  $\g=p_{I\cap J}(p)=p_{I\cap J}(p_J(p))=p_J(p)$.
  And so $\g\in p_J(P)\cap\G_{I\cap J}$.
  
  Conversely, let $\g\in p_J(P)\cap\G_{I\cap J}$.
  Pick a $p\in P$ with $p_J(p)=\g$. Note that
  $p_{I\cap J}(p)=p_{I\cap J}(p_J(p))=p_{I\cap J}(\g)=\g$.
  But $p_J(p)\in\G_{I\cap J}$, so $p\in\G_I$ and hence
  $\g\in p_{I\cap J}(P\cap\G_I)$.
\end{proof}

\begin{lemma}
  Let $P\leq\G_1\times\dots\times\G_n$. If $P$ virtually\label{lem:vskernel}
  surjects to $k$-tuples (for some $k\geq 2$) then
  $N_{1,\dots,n-1}=P\cap(\G_1\times\dots\times\G_{n-1})$ virtually surjects to 
  $(k-1)$-tuples.
\end{lemma}
\begin{proof}
  Let $J'\subset\{1,\dots,n-1\}$ be some $(k-1)$-element subset.
  Apply the preceding lemma with $I:=\{1,\dots,n-1\}$ and
  $J:=J'\cup\{n\}$ to conclude that
  $p_{J'}(N_{1,\dots,n-1})=p_J(P)\cap\G_{J'}$. Note that
  $p_J(P)$ has finite index in $\G_J$, so $p_J(P)\cap\G_{J'}$
  has finite index in $\G_{J'}$.
\end{proof}

\begin{thm}
  Assume that the $n$-$(n+1)$-$(n+2)$ Conjecture holds
  whenever the quotient $Q$ is virtually nilpotent.
  \label{thm:implication}
  
  Then the Virtual Surjections Conjecture is true.
\end{thm}
\begin{proof}
  We prove the following statement by induction on $k\geq 1$:
  \begin{unnumclaim}
  Let $\G_1,\dots,\G_n$ be groups of type $\Fin_k$
  and $P\leq\G_1\times\dots\times\G_n$
  a subgroup of the direct product. If $P$ virtually
  surjects to $k$-tuples and $p_i(P)/(P\cap\G_i)$ is virtually nilpotent
  for all $i$, then
  $P$ is of type $\Fin_k$.
  \end{unnumclaim}
  
  Of course, the nilpotency assumption is redundant when $k\geq 2$ by Lemma \ref{lem:nilp},
  so this Claim does imply the Virtual Surjections Conjecture.
  Furthermore, by the same argument as in the proof of Proposition
  \ref{prop:vs1} we can always reduce to the case where $P$
  is a subdirect product: If $P$ is not subdirect, we replace all the $\G_i$
  by their finite index subgroups $p_i(P)$, which will not affect the finiteness conditions.
  
  We have already supplied the base case $k=1$ of the induction in Proposition \ref{prop:vs1}.
  So fix some $k\geq 2$. We now prove the Claim for this
  $k$ by induction on $n$.
  
  If $n\leq k$ then there is nothing to prove: $P$ will be
  of finite index in $\G_1\times\dots\times\G_n$ and therefore
  certainly of type $\Fin_k$.
  
  So assume $n>k$. Set $T:=p_{1,\dots,n-1}(P)$, so $P$
  is the fibre product associated to the sequences
  \begin{alignat*}{2}
  \ses{N_{1,\dots,n-1}}{&\,T&}Q\\
  \ses{N_n}{&\,\G_n\,&}Q.
  \end{alignat*}
  Here, $T\leq \G_1\times\dots\times\G_{n-1}$ is a subdirect product, which
  also virtually surjects to $k$-tuples, since $p_J(T)=p_J(P)$
  for all $J\subset\{1,\dots,n-1\}$. By induction, we can assume that
  $T$ is of type $\Fin_k$.
  
  By Lemma \ref{lem:vskernel}, $N_{1,\dots,n-1}\leq\G_1\times\dots\times\G_{n-1}$
  virtually surjects to $(k-1)$-tuples. Furthermore,
  \[\quot{\G_i}{(N_{1,\dots,n-1}\cap\G_i)}=\quot{\G_i}{P\cap\G_i}\]
  is virtually nilpotent for all $i\in\{1,\dots,n-1\}$. Thus, again by induction
  (on $k$ this time), we can assume that $N_{1,\dots,n-1}$ is of type 
  $\Fin_{k-1}$.
  
  By assumption, $Q\isom\G_n/N_n$ is virtually nilpotent
  and finitely generated (since $\G_n$ is) and therefore of type $\Fin_\infty$.
  So, assuming the $(k-1)$-$k$-$(k+1)$ Conjecture, we can conclude
  that $P$ is of type $\Fin_k$.
\end{proof}

Two remarks are in order. Firstly, from the proof it is clear that
we only need to assume the $n$-$(n+1)$-$(n+2)$ Conjecture for
$n<k$ to conclude that virtual surjection to $k$-tuples
implies type $\Fin_k$. In particular, the $1$-$2$-$3$ Theorem
implies the Virtual Surjection to Pairs Theorem, i.e.\ if a 
subgroup $P\leq\G_1\times\dots\times\G_n$ of a direct product
of finitely presented groups $\G_1,\dots,\G_n$ virtually
surjects to pairs of factors $\G_i\times \G_j$, then $P$ is
finitely presented. This was first proved in \cite[Theorem E]{BHMSFinPres}.

The second remark concerns the word ``virtually'' in the statement
of Theorem \ref{thm:implication}. In fact, a quick argument shows that
it suffices to assume the $n$-$(n+1)$-\mbox{$(n+2)$} Conjecture for nilpotent
quotients $Q$, as this implies the conjecture for \emph{virtually}
nilpotent quotients.

To see this, assume that the $n$-$(n+1)$-$(n+2)$ Conjecture holds
for some particular quotient group $Q$ and let $\bar Q$ be a group
that contains $Q$ as a finite index subgroup.

Let
\begin{gather*}
\ses[\pi_1]{N_1}{\G_1}{\bar Q}\\
\ses[\pi_2]{N_2}{\G_2}{\bar Q}
\end{gather*}
be short exact sequences such that $N_1$ is of type $\Fin_n$ and
$\G_1$ and $\G_2$ are of type $\Fin_{n+1}$. We will show that the fibre
product of $\pi_1$ and $\pi_2$, call it $P$, is of type $\Fin_{n+1}$.
Set $\G_1':=\pi_1^{-1}(Q)$ and $\G_2':=\pi_2^{-1}(Q)$,
so we have another pair of short exact sequences
\begin{gather*}
\ses[\pi_1|]{N_1}{\G_1'}{Q}\\
\ses[\pi_2|]{N_2}{\G_2'}{Q}.
\end{gather*}
Now $\G_1'$ and $\G_2'$ have finite index in $\G_1$ and $\G_2$ respectively,
so they are of type $\Fin_{n+1}$.
We can therefore apply
the assumption that the $n$-$(n+1)$-$(n+2)$ Conjecture holds
for $Q$ to obtain that $P'\leq\G_1'\times\G_2'$, the fibre product
associated to the latter pair of short exact sequences, is of type $\Fin_{n+1}$.
But $P'$ is just the intersection of $P$ with the finite index subgroup
$\G_1'\times\G_2'\leq\G_1\times\G_2$, so $P'$ has finite index in $P$.
Thus $P$ is of type $\Fin_{n+1}$ as well.

\section{A topological approach}\label{sec:top}

The statement of the \nonetwo\ Conjecture is at its core a topological question,
as the finiteness properties $\Fin_n$ are concerned with the existence of 
classifiying complexes of a certain type and thus implicitly with
questions about the finite generation of higher homotopy groups:
Having constructed a finite $n$-skeleton $\skel Xn$ of a classifying complex
for a group $\G$, the group is of type $\Fin_{n+1}$ if and only if
$\pi_n(\skel Xn)$ is finitely generated as a $\Z\G$-module.

It thus appears natural to approach the Conjecture from a topological
point of view. Indeed, that is the route 
taken in the proof of the $1$-$2$-$3$ Theorem in \cite{BBMS}.
There, the authors make heavy use of the combinatorial characterisation
of the second homotopy group of a $2$-complex in terms
of identity sequences and crossed modules. Such a characterisation, alas,
is not available for the higher homotopy groups, and so
their method of proof cannot be emulated in the general case in any obvious way.

Yet, we will see how even without the combinatorial
formalism, we can obtain a number of results from a topological setting. We will derive the 
Split $n$-$(n+1)$-$(n+2)$ Conjecture
from a general lemma relating the finiteness properties
of the groups in two short exact sequences. Furthermore we will show
that the general $n$-$(n+1)$-$(n+2)$ Conjecture can be reduced to the
case where the second factor is a finitely generated free group.

\subsection{Introducing stacks}
To start off our investigation of fibre products we first need
to establish some tools for constructing classifying complexes
of group extensions. Given a short exact sequence of groups
$\ses N\G Q$ along with classifying complexes for the kernel and
quotient, the goal is to construct a useful classifying complex
for the middle group $\G$. ``Useful'' is a vague term here, of course,
and various different decriptions can be found in
the literature (e.g.\ the
``complexes of spaces'' in \cite{CorsonComplexesOfSpaces},
the construction in \cite{Stark} and the more algebraic
analogue of ``complexes of groups'' 
\cite{CorsonComplexesOfGroups,HaefligerComplexesOfGroups}).
The best-suited to our purposes, and the one we will employ 
here, is the one given by Geoghegan in terms
of ``stacks'' in \cite[Chapter 6]{GeogheganBook}.

Let $\ses[\pi] N\G Q$ be a short exact sequence of groups and
$X$ a $\B N$-complex, $Z$ a $\B Q$-complex. Then Geoghegan's construction
yields a $\B\G$-complex $Y$ and a cellular map $Y\to Z$, which
satisfies properties similar to those of a fibre product over $Z$ with fibre $X$.
We shall sketch the relevant points of the construction.
First note that we can use the Borel construction 
to obtain a $\B\G$-complex $Y_1$, which is an
actual fibre product over $Z$: Let
$Y_0$ be any $\B\G$-complex, and let $\tilde Y_0$, $\tilde Z$ denote the
universal covers of $Y_0$ and $Z$, respectively. $\G$ acts naturally
on $\tilde Y_0$ as the group of deck transformations
and also on $Z$ via the homomorphism $\pi:\G\to Q$.
So $\G$ acts diagonally on the direct product $\tilde Y_0\times\tilde Z$,
and this action is free, since the action on the first factor is.
Therefore the quotient $Y_1:=(\tilde Y_0\times\tilde Z)/\G$ is a $\B\G$-complex.
Furthermore the projection map $\tilde Y_0\times\tilde Z\to\tilde Z$
descends to a cellular map $\phi_1:Y_1\to Z$ on the respective quotients by the $\G$-actions.
It is easily checked that $\phi_1$ is in fact a fibre bundle
with fibre $\tilde Y_0/N$. This fibre is a $\B N$-complex, though usually
a very large one. We would like to replace the fibres by some given
$\B N$-complex $X$, e.g.\ in our case, one with finite skeleta up to some dimension.
Geoghegan's construction yields such a complex, but it will not in general
be a fibre bundle over $Z$, but merely a \emph{stack}, defined as follows
in \cite[Section 6.1]{GeogheganBook}:
\begin{defn}
  Let $F$, $E$ and $B$ be CW-complexes. Denote by $C_n$ the set of closed $n$-cells
  of $B$. A map $\phi:E\to B$ is called a stack with fibre $F$, if
  there is for each $c\in C_n$ a map $f_e:F\times S^{n-1}\to\phi^{-1}(\skel B{n-1})$
  and denoting the disjoint union of these maps by\label{defn:stack}
  \[f_n:\bigdisjcup_{c\in C_n} X\times S^{n-1}\to \phi^{-1}(\skel B{n-1})\]
  there are for each $n\geq 0$ homeomorphisms\footnotemark
  \[k_n:\phi^{-1}(\skel B{n-1})\cup_{f_n}\left(\bigdisjcup_{c\in C_n} X\times D^n\right)\to\phi^{-1}(\skel Bn),\]
  such that
  \begin{enumerate}
    \item $k_n$ agrees with inclusion on $\phi^{-1}(\skel B{n-1})$.
    \item $k_n$ maps each closed cell of $E$ onto a closed cell of $B$.
    \item The following diagram commutes:
    \[\begin{tikzcd}
       \phi^{-1}(\skel B{n-1})\disjcup\left(\bigdisjcup_{c\in C_n} F\times D^n\right)\dar{\text{quotient}}\drar{u}\\
       \phi^{-1}(\skel B{n-1})\cup_{f_n}\left(\bigdisjcup_{c\in C_n} F\times D^n\right)\dar{k_n}&\skel Bn\\
       \phi^{-1}(\skel Bn)\ar{ur}{\phi|}
    \end{tikzcd}\]
    where $u$ agrees with $\phi$ on $\phi^{-1}(\skel B{n-1})$ and on each
    copy $X\times D^n$ corresponding to some cell $c\in C_n$ it is projection to $D^n$
    followed by a characteristic map $D^n\to c$.
  \end{enumerate}
\end{defn}
\footnotetext{We write $\cup_{f_n}$ for ``gluing along $f_n$'', i.e.\ the quotient
space of the disjoint union identifying each $x$ in the domain of $f_n$ with $f_n(x)$.}
Roughly speaking, if $\phi:E\to B$ is a stack, then $E$ can be built up step-by-step as a filtration
$E_n:=\phi^{-1}(\skel Bn)$, where going from $E_{n-1}$ to $E_n$ involves
gluing in a complex $F\times D^n$ over each $n$-cell $c\in C_n$.
Over the interior $\mathring c$ of the cell the map 
$\phi:\phi^{-1}(\mathring c)\to \mathring c$
has the form of a projection $F\times\mathring c\to\mathring c$.

Note also what the definition means in dimension $0$: By convention,
$S^{-1}=\emptyset=\skel B{-1}$, so $f_0$ is the map $\emptyset\to\emptyset$ and
$k_0$ is a homeomorphism $\bigdisjcup_{e\in E_0} F\to\phi^{-1}(\skel B0)$, that is
$\phi^{-1}(\skel B0)$ consists of a copy of $F$ for each vertex of $B$ and
$\phi$ merely maps each copy to the corresponding vertex.

The utility of Definition \ref{defn:stack} lies in two facts: Firstly, the Borel construction
described above in fact yields a stack $(\xtilde Y_0\times\xtilde Z)/\Gamma\to Z$
with fibre $X_0:=\xtilde Y_0/N$.
Secondly, if $\phi:E\to B$ is a stack with fibre $F$ and $F'$ is homotopy equivalent to $F$, 
we can ``rebuild'' the stack, by gluing in the fibre $F'$ instead of $F$ at each step.
This yields a new stack $\phi':E'\to B$, and a homotopy equivalence
$h:E\to E'$ making the diagram
\[\begin{tikzcd}[column sep=small]
E\drar[swap]{\phi}\ar{rr}{h}&&E'\dlar{\phi'}\\
&B
\end{tikzcd}\]
commute up to homotopy. This homotopy can be chosen such that
the above diagram commutes up to homotopy
``over every cell'', meaning that there is a
homotopy $H:E\times I\to B$ from $\phi$ to $\phi'\comp h$
with $H(\phi^{-1}(e)\times I)\subset e$ for all closed
cells $e\subset B$.
Proofs of these facts can be found in \cite[Section 6.1]{GeogheganBook}.

In particular, in our example of a stack obtained from the Borel construction,
we can replace the fibre $X_0$ by any given $\B N$-complex $X$.
So we obtain the following (this is \cite[Theorem 7.1.10]{GeogheganBook}):
\begin{thm}
Let\label{thm:stack} $\ses[\pi]N\G Q$ be a short exact sequence. Furthermore, let
$X$ be a $\B N$-complex and $Z$ a $\B Q$-complex and identify their
fundamental groups with $N$ and $Q$, respectively, by fixing isomorphisms
$N\isom\pi_1(X,x_0)$ and $Q\isom\pi_1(Z,z_0)$. Then there is a
$\B\G$-complex $Y$ and a cellular map $\phi:Y\to Z$ with the following
properties:
\begin{enumerate}
\item $\phi$ induces the homomorphism $\pi:\Gamma\to Q$ on the level of fundamental
groups.
\item $\phi$ is a stack with fibre $X$.\qed
\end{enumerate}
\end{thm}

\subsection{Stacks and an $\Fin_{n+1}$-criterion}

Note that in any stack $\phi:E\to B$ with fibre $F$ the $m$-cells of $E$ are in bijective
correspondence with the set
$\bigcup_{k+l=m} C_k^B\times C_l^F$
where $C_k^B$ denote the set of $k$-cells of $B$ and $C_l^F$ the set of $l$-cells of $F$.
So in particular, if $B$ and $F$ have finite $n$-skeleta then the same will be true of $E$.
Applying this to the stack supplied by Theorem \ref{thm:stack}, we obtain the well-known
statement that any extension
of a group of type $\Fin_n$ by a group of type $\Fin_n$ is of type $\Fin_n$ itself.

Let us now consider a short exact sequence of groups
\[\ses[\pi]NP\G\]
where $N$ is of type $\Fin_n$ and $\G$ is of type $\Fin_{n+1}$ for some $n\geq 2$.\footnote{Most
of what follows can, in principle, be applied to the case $n=1$ as well. However, we have treated this case
before and excluding it here sometimes avoids clutter.}
We shall look for conditions for $P$ to be of type $\Fin_{n+1}$ (later, we will apply these considerations
to the short exact sequence $\ses{N_1}P{\Gamma_2}$ associated
to a fibre product $P\leq\Gamma_1\times\Gamma_2$).

Pick a $\B N$-complex $X$ with finite $n$-skeleton and a $\B\G$-complex
$Y$ with finite $(n+1)$-skeleton and let $\phi:W\to Y$ be a stack with fibre $X$,
as given by Theorem \ref{thm:stack}.
That is, $W$ is a $\B P$-complex and $\phi$ induces the homomorphism $\pi:P\surjarrow \G$.
We may furthermore assume that $Y$ has a single vertex $\ast$ (e.g.\ \cite[Proposition 7.1.5]{GeogheganBook}).
By the stack property, $\phi^{-1}(\ast)$ is a subcomplex isomorphic to $X$. 
Identifying these complexes,
we will frequently regard $X$ as a subcomplex of $W$.

As remarked above, the $n$-skeleton of $W$ will be finite.
The $(n+1)$-skeleton of $W$ will usually be infinite, since there might be
infinitely many $(n+1)$-cells in $X$. However, there are only finitely many
$(n+1)$-cells in $W$ that project down to cells in $Y$ of dimension $\geq 1$. 
In other words,
only finitely many $(n+1)$-cells of $W$ are not completely contained in $X=\phi^{-1}(\ast)$. So
$P$ is of type $\Fin_{n+1}$ if and only if we can remove from the $(n+1)$-skeleton
$\skel W{n+1}$ all but finitely many of the $(n+1)$-cells in $X\leq W$ while making sure that
the complex stays $n$-aspherical.

We will prove the following criterion, comparing the finiteness properties
in two short exact sequences related by a homomorphism:
\begin{prop}
Let\label{prop:transferfn} $\ses NP\G$ be a short exact sequence of groups with $\G$ of type
$\Fin_{n+1}$ and $N$ of type $\Fin_n$ for some $n\geq 2$. 
Assume that there is another short exact sequence $\ses{N'}{P'}{\G'}$ and a homomorphism
$\phi:P'\to P$ which restricts to an isomorphism $\phi:N'\to N$
such that $P'$ is of type $\Fin_{n+1}$.
\[\begin{tikzcd}
{\scriptstyle \Fin_n}&{\scriptstyle \Fin_{n+1}}&{\scriptstyle \Fin_{n+1}}\\[-5ex]
N'\rar[hook]\dar[hook, two heads]{\phi|}&P'\rar[two heads]\dar{\phi}&\G'\dar\\
N\rar[hook]&P\rar[two heads]&\G\\[-5ex]
{\scriptstyle \Fin_n}&&{\scriptstyle \Fin_{n+1}}
\end{tikzcd}\]
Then $P$ is of type $\Fin_{n+1}$ as well.
\end{prop}

First we will need an easy lemma which essentially states that removing some $(n+1)$-cells
from an $n$-aspherical complex results in a complex whose $n$\textsuperscript{th} homotopy
group is generated, as a $\pi_1$-module, by the attaching maps of the $(n+1)$-cells that have
been removed.

\begin{lemma}
Let $n\geq 2$. Let $A$ be a\label{lem:wh} connected, $(n+1)$-dimensional
CW-complex with $\pi_n(A)=0$
and $B,C\subset A$ subcomplexes
such that $C$ is connected, $\skel Bn=\skel An$ and $B\cup C=A$. Then $\pi_n(B)$ is generated
as a $\pi_1(B)$-module by the image of the map $\pi_n(\skel Cn)\to\pi_n(B)$
induced by inclusion.
\end{lemma}
\begin{proof}
The pair $(\skel A{n+1},\skel An)$ is an $(n+1)$-cellular extension, so by a theorem of J.H.C.\ Whitehead
(\cite{WhiteheadAddingRelations,WhiteheadCHII}; see also \cite[Chapter V.1, Theorem 1.1]{WhiteheadHomotopyBook})
the kernel of the map $\pi_n(\skel An)\to\pi_n(\skel A{n+1})$ induced by inclusion is generated
by the attaching maps of $(n+1)$-cells. To make this statement precise, we will have to introduce basepoints.
For each $(n+1)$-cell $e$ of $A$, 
pick a characteristic
map $h_e:(D^{n+1},S^n)\to (\skel A{n+1},\skel An)$. Then pick basepoints $v_0\in \skel A0$ and
$x_0\in S^n$ and for each $e$ a path $p_e:[0,1]\to A$ from $v_0$ to $h_e(x_0)$. Writing
$\eps_n$ for some fixed generator of $\pi_n(S^n,x_0)$ and $d_e$ for the restriction
$h_e|_{S^n}$ (i.e.\ $d_e$ is an attaching map for the cell $e$), we have that $d_e$
represents an element ${d_e}_*({\eps_n})\in\pi_n(\skel An,h_e(x_0))$ and $p_e$
induces an isomorphism
\[\tau_{p_e}:\pi_n(\skel An,h_e(x_0))\to\pi_n(\skel An,v_0).\] The theorem by Whitehead
alluded to above then says that
the elements $\tau_{p_e}{d_e}_*(\eps_n)$ generate 
$\pi_n(\skel An,v_0)$ as a $\pi_1(A,v_0)$-module.

Now $(\skel An,B)$ is again an $(n+1)$-cellular extension, so $\iota_*:\pi_n(\skel An,v_0)\to\pi_n(B,v_0)$,
induced by inclusion, is a surjection. Since $B$ and $A$ have the same $2$-skeleton, we have
$\pi_1(A)=\pi_1(B)$. Therefore the elements $\iota_*\bigl(\tau_{p_e}{d_e}_*(\eps_n)\bigr)$
generate $\pi_n(B,v_0)$ as a $\pi_1(B)$-module. Now $\iota_*\bigl(\tau_{p_e}{d_e}_*(\eps_n)\bigr)=0$
for those $(n+1)$-cells $e$ which are contained in $B$, and all the other $(n+1)$-cells are contained in $C$,
since $B\cup C=A$. But if $e$ is in $C$ then $\iota_*\bigl(\tau_{p_e}{d_e}_*(\eps_n)\bigr)$ lies in
the image of the map $\pi_n(C^{(n)},v_0)\to\pi_n(B,v_0)$ (we can always choose $p_e$ to lie in $C$),
so this image generates $\pi_n(B,v_0)$ as a $\pi_1(B,v_0)$-module. This proves
the assertion.
\end{proof}

\begin{thm}
Let $\ses[\pi]NP\G$ be an exact sequence of groups, such
that $\G$ is of type\label{thm:contract}
$\Fin_{n+1}$ and $N$ is of type $\Fin_n$ (where $n\geq 2$).
Furthermore let $X$ be a $\B N$-complex with finite $n$-skeleton
and fix an isomorphism to identify $\pi_1(\B N)$ with $N$.
Then $P$ is of type $\Fin_{n+1}$ if and only if there exists a finite 
$(n+1)$-dimensional, $n$-aspherical
complex $\bar X$, such that
\begin{enumerate}
  \item $\skel Xn$ can be embedded as a subcomplex in $\bar X$
    and the inclusion $\iota:\skel Xn\injarrow \bar X$ induces an
    injective map $\iota_*$ on fundamental groups.
  \item There is a homomorphism $f:\pi_1(\bar X)\to P$ making the following
    diagram commute:
    \[\begin{tikzcd}
      \skel Xn\rar{\iota}&\bar X\\[-5ex]
      \pi_1(\skel Xn)\rar[hook]{\iota_*}\dar[equals]&\pi_1(\bar X)\dar{f}\\
      N\rar[hook]&P
    \end{tikzcd}\]
\end{enumerate}
\end{thm}
\begin{proof}
Let $Y$ be a $\B\G$-complex with a single vertex and finite $(n+1)$-skeleton and let
$\phi:W\to Y$ be a stack with fibre $X$ and inducing $\pi$ on fundamental
groups, as per Theorem \ref{thm:stack}. Then $W$ has finite $n$-skeleton. As before we regard
$X=\phi^{-1}(Y)$ as a subcomplex of $W$.

To prove one direction of the claim, suppose $P$ is of type $\Fin_{n+1}$. In this case some finite
$P$-invariant subcomplex $\skel{W'}{n+1}\subset \skel W{n+1}$ is $n$-aspherical.
Clearly $\skel{W'}{n+1}$ contains $X^{(n)}$ as a $\pi_1$-embedded
subcomplex and the inclusion $\skel{W'}{n+1}\to \skel W{n+1}$ induces a homomorphism
$f$ as required in the claim.

The idea of the proof for the other direction is to ``replace'' the potentially infinitely many $(n+1)$-cells of $X\subset W$ 
with the finitely many
$(n+1)$-cells from $\bar X$, glued into $W$ using a topological realisation of $f$.

Let $V$ denote the subcomplex of $\skel W{n+1}$ consisting of $\skel Wn$ together with those
finitely many
$(n+1)$-cells of $W$ that are mapped under $\phi$ to cells of $Z$ of dimension $\geq 1$.
Applying Lemma \ref{lem:wh} with $A=\skel W{n+1}$, $B=V$ and $C=\skel X{n+1}$ we obtain that
$\pi_n(V)$ is generated
as a $\pi_1(V)$-module by the image of the map $\pi_n(\skel Xn)\to\pi_n(V)$
induced by inclusion. Roughly speaking, we can push any $n$-sphere in $W$ into $\skel Xn$
via a homotopy not intersecting any of the $(n+1)$-cells in $X$. We now proceed to make
$n$-spheres in $X$ contractible.

We have the following diagram of spaces and inclusions and the induced
diagram on fundamental groups:
\[\begin{tikzcd}
\skel Xn\dar[equals]\rar[hook]{\iota}&\skel{\bar X}n   &&   \pi_1(\skel Xn)\dar[equals]\rar{\iota_*}&\pi_1(\skel{\bar X}n)\dar[dashed]{f}\\
\skel Xn\rar[hook]&\skel Wn        &&  N\rar[hook]&P
\end{tikzcd}\]
Since $\skel Wn$ is $(n-1)$-aspherical there is no obstruction to constructing a cellular map
$\psi:\skel{\bar X}{n}\to \skel Wn$ making the left square commute and inducing $f$ on the level of fundamental groups.
Now let $\xtilde W:=V\cup_{\psi}\bar X$,
i.e.\ $\xtilde W$ is
constructed by attaching to $V$ the $(n+1)$-cells of $\bar X$ along $\psi$.
Denote by $\bar\psi:\bar X\to\xtilde W$ the natural map (i.e.\ the composition of the inclusion map $\bar X\to V\disjcup \bar X$
with the quotient map $V\disjcup \bar X\to V\cup_{\psi}\bar X$).
\[\begin{tikzcd}
\skel Xn\dar[equals]\rar[hook]&\skel{\bar X}n\dar{\psi}\ar[hook]{rr}&&\bar X\dar{\bar\psi}\\
\skel Xn\rar[hook]&\skel Wn\rar[hook]&V\rar[hook]&\xtilde W
\end{tikzcd}\]

So $\xtilde W$ is a finite, $(n+1)$-dimensional complex with the same $n$-skeleton
as our original $\B P$-complex $W$. We will now show that $\xtilde W$ is in fact
$n$-aspherical, and therefore a finite $\B[n+1]P$, proving that $P$
is of type $\Fin_{n+1}$.

We consider the following sequence of homomorphisms induced by inclusions:
\[\pi_n(\skel Xn)\stackrel{\iota_1}\to\pi_n(V)\stackrel{\iota_2}\to\pi_n(\xtilde W).\]
First note that, since $\skel Vn=\skel{\xtilde W}n$, the map $\iota_2$ is a surjection.
If $\alpha:S^n\to \skel Xn$ is any $n$-sphere in $X$ then, since $\bar X$ is $n$-aspherical, $\alpha$ can be contracted
in $\bar X$, i.e.\ there is an extension
$\hat\alpha:D^{n+1}\to\bar X$ with $\hat\alpha|_{S^n}=\alpha$.
Since $\psi$ was assumed to restrict to inclusion on $\skel Xn$, so does
$\bar\psi$. So the composition
$\bar\psi\comp\hat\alpha:D^{n+1}\to\xtilde W$ shows that
$\alpha$ is null-homotopic in $\xtilde W$. In other words, the map
$\iota_2\comp\iota_1:\pi_n(\skel Xn)\to\pi_n(\skel{\xtilde W}{n+1})$, induced by inclusion,
is the zero map.

But we have shown above that the image of $\iota_1$ generates $\pi_n(V)$
as a $\pi_1$-module. Therefore a set of generators of $\pi_n(V)$ is sent to $0$
under $\iota_2$, proving that $\iota_2$ is the zero map. Combining this with the fact
that $\iota_2$ is a surjection, we get that $\pi_n(\xtilde W)=0$, as claimed.
\end{proof}

Proposition \ref{prop:transferfn} now follows as a corollary.

\begin{proof}[Proof of Proposition \ref{prop:transferfn}]
  Let $X$ be a $\B N'$-complex with finite $n$-skeleton. The quotient
  $\G'\isom P'/N'$ is of type $\Fin_n$ (actually, even $\Fin_{n+1}$; see
  e.g.\ \cite[Theorem 7.2.21]{GeogheganBook}), so
  we can pick
  a $\B\G'$-complex $Z'$ with finite $n$-skeleton and
  a single vertex as its $0$-skeleton. 
  
  By Theorem \ref{thm:stack}
  we can pick a $\B P'$-complex $W$ which is a stack over $Z'$ with fibre $X$.
  In particular $W$ has finite $n$-skeleton and contains $X$ as a subcomplex. 
  Since $P'$ is of type
  $\Fin_{n+1}$, there is a finite $n$-aspherical subcomplex $\skel{W'}{n+1}\subset\skel W{n+1}$
  with $\skel{W'}n=\skel Wn$. Since the fundamental group depends only on the $2$-skeleton,
  we have $\pi_1(\skel{W'}{n+1})=\pi_1(\skel W{n+1})\isom P'$ and
  $\skel{W'}{n+1}$ still contains $\skel Xn$ as a subcomplex with $\pi_1$-injective
  embedding $\skel Xn\injarrow \skel{W'}{n+1}$. Thus we can take
  $\skel{W'}{n+1}$ to be $\bar X$ in Theorem \ref{thm:contract}
  and conclude that $P$ is of type $\Fin_{n+1}$.
\end{proof}

\subsection{The Split $n$-$(n+1)$-$(n+2)$ Conjecture and the reduction to free $\G_2$}
Applying this to the situation of the \nonetwo\ Conjecture, consider, as always, two short 
exact sequences
\begin{gather*}
\ses[\pi_1]{N_1}{\G_1}Q\\
\ses[\pi_2]{N_2}{\G_2}Q
\end{gather*}
such that $N_1$ is of type $\Fin_n$ and $\Gamma_1$ and $\Gamma_2$ are of type $\Fin_{n+1}$
(with $n\geq 2$).

Then the associated fibre product $P$ fits into a short exact sequence
\[\ses[p_2]{N_1}P{\G_2},\]
which satisfies the finiteness requirements of Proposition \ref{prop:transferfn}. So to show that
$P$ is of type $\Fin_{n+1}$ we need to find a group $P'$ as in the statement of the Proposition. 
Note that $\Gamma_1$
is a group of type $\Fin_{n+1}$ that contains $N_1$ as a normal subgroup, as required of $P'$.
However, in general there is no homomorphism $\phi:\Gamma_1\to P$ which maps
$N_1\leq\G_1$ to $N_1\times 1\leq P$.\footnote{Note that in the symmetric case, where
the two short exact sequences are identical, there always is a natural homomorphism 
$\Gamma_1\to P\leq\Gamma_1\times\Gamma_1$, sending $\gamma$ to $(\gamma,\gamma)$.
But this homomorphism does not fit into a commutative 
diagram, as in the statement of Proposition \ref{prop:transferfn}, since it
does not map $N_1\leq \G_1$ to $N_1\times 1\leq P$.}
Yet there is one particular special case, where such a map does always exist, namely when the
second short exact sequence splits. This allows us to prove the split case
of the $n$-$(n+1)$-$(n+2)$ Conjecture.
\begin{cor}
Let\label{cor:splitn12}
\begin{gather*}
\ses[\pi_1]{N_1}{\G_1}Q\\
\ses[\pi_2]{N_2}{\G_2}Q
\end{gather*}
be short exact sequences with $\Gamma_1$, $\Gamma_2$ of type $\Fin_{n+1}$
and $N_1$ of type $\Fin_n$ and assume the second sequence splits.
Then the associated fibre product is of type $\Fin_{n+1}$.
\end{cor}
\begin{remark}
  Note that we do not need an additional
assumption on $Q$ here (though the ones we have made on $\G_1$ and $N_1$
imply that $Q$ is at least of type $\Fin_{n+1}$).
\end{remark}
\begin{proof}
Let $\sigma:Q\to\Gamma_2$ be a splitting
homomorphism for the second sequence, so $\pi_2\comp\sigma=\id_Q$.
Now consider the exact sequence
\[\ses[p_2]{N_1}P{\Gamma_2}.\]
Define
\[\phi:\Gamma_1\to P, \gamma_1\mapsto(\gamma_1,\sigma\pi_1(\gamma_1)).\]
Note that since $\pi_1(\gamma_1)=\pi_2\sigma\pi_1(\gamma_1)$ the image of $\phi$ does indeed lie in $P$.
Obviously, $\phi$ restricts to the identity on $N_1$.
Therefore, Proposition \ref{prop:transferfn} applies to yield the conclusion.
\end{proof}

This result was previously obtained, using
different methods, by Martin Bridson and Ken Brown in unpublished work.

As for the $1$-$2$-$3$ Theorem, we can use Proposition
\ref{prop:transferfn} not just to prove the split case but also
to reduce the general case of the $n$-$(n+1)$-$(n+2)$ Conjecture
to the case when $\G_2$ is finitely generated free,
by the same method as before.

Consider again the sequences
\begin{gather*}
\ses[\pi_1]{N_1}{\G_1}Q\\
\ses[\pi_2]{N_2}{\G_2}Q
\end{gather*}
with $N_1$ of type $\Fin_n$, $\Gamma_1$ and $\Gamma_2$ of type $\Fin_{n+1}$ and
$Q$ of type $\Fin_{n+2}$
Now let $F$ be a finitely generated free group with an
epimorphism $p:F\to\Gamma_2$. Let $P'$ denote the fibre
product associated to the short exact sequences
\begin{alignat*}{3}
N_1&\injarrow\Gamma_1&&\overset{\pi_1}\surjarrow Q\\
N_2'&\injarrow F&&\overset{\pi_2p}\surjarrow Q
\end{alignat*}
where $N_2'=\ker(\pi_2p)=p^{-1}(N_2)$. Now the homomorphism
\[\id\times p:\Gamma_1\times F\to\Gamma_1\times\Gamma_2,(\gamma_1,g)\mapsto(\gamma_1,p(g))\]
maps $P'$ onto $P$ and clearly sends $N_1\times 1$ to $N_1\times 1$.
Thus by Proposition \ref{prop:transferfn}, $P$ is of type $\Fin_{n+1}$ if
$P'$ is.

\begin{prop}
  \label{prop:reductiontofree}
  If the $n$-$(n+1)$-$(n+2)$ Conjecture holds whenever $\G_2$
  is finitely generated free then it holds in general. 
\end{prop}

\section{Homological finiteness properties of fibre products}
\label{sec:weak}
In the topological approach presented in the previous section,
we endeavoured to prove the $n$-$(n+1)$-$(n+2)$ Conjecture by
constructing a finite $(n+1)$-skeleton of a classifying space
of the fibre product directly. We have seen how proving a group
to be of type $\Fin_{n+1}$ comes down to proving a certain higher homotopy
group to be finitely generated.

Here, we will take a different course, and consider
homology groups instead of homotopy. This approach brings with it
the usual advantages of homology over homotopy: Better computability
and the availability of powerful algebraic tools, notably, in our case,
spectral sequences. It also brings the usual disadvantages: The
homological finiteness properties we will consider here are strictly
weaker than the $\Fin_n$-properties.

Nonetheless, these properties have proved useful in the analysis of subdirect
products in the past. We will prove in this section homological versions of the
$n$-$(n+1)$-$(n+2)$ and Virtual Surjection Conjectures, which,
unlike the result on split sequences from the previous section, 
apply in full generality.

\subsection{Type $\wFP_n$}
\label{subsec:homfin}

We will consider the following homological finiteness condition.\footnote{A similar
condition has been used before by K.\ Brown in \cite{BrownEuler2,BrownEuler3} in the
context of Euler characteristics of groups. He defines
a group to be of finite homological type if it has finite virtual
cohomological dimension and every finite index subgroup
has finitely generated integral homology. 
We use a different nomenclature to avoid confusion with another,
stronger condition, also called ``finite homological type'', used in \cite{BrownBook},
see also \cite{LearySaadetoglu}.}

\begin{defn}
  A group $\G$ is of type $\wFP_n$ (``weak $\FP_n$'') if
  the homology groups 
  $\Hom_k(\G',\Z)$ are finitely generated for all
  $k\leq n$ and all finite index subgroups $\G'\leq\G$.
\end{defn}

Recall that a group is of type $\FP_n$ if $\Z$, considered
as a trivial $\G$-module, has a free resolution by $\G$-modules
which is finitely generated up to dimension $n$. Any group of
type $\Fin_n$ is of type $\FP_n$ (the augmented
chain complex of the universal cover of a suitable
classifying complex supplies the desired resolution). And, since
having type $\FP_n$ is preserved under passing to subgroups of finite index,
every group of type $\FP_n$ is of type $\wFP_n$. The converse
of this latter implication does not hold, as shown by
Leary and Saadeto\u{g}lu \cite{LearySaadetoglu}.\footnote{The authors
employ Brown's stronger definition of finite homological type. See Remark 2
in their paper for the relationship between these two conditions.} 

\subsection{The Weak $n$-$(n+1)$-$(n+2)$ Theorem and the Weak Virtual Surjections Conjecture}
Let $\ses N\G Q$ be a short exact sequence.
Using the LHS spectral sequence it is easy to prove that\footnote{we will consider
only integral homology here, so we write
$\Hom_n(\G):=\Hom_n(\G,\Z)$}
$\Hom_n(\G)$ will be finitely generated
if $Q$ is of type $\FP_n$ and $N$ is of type $\wFP_n$. 
Looking a little more carefully, we see that 
in fact, we can get away with a weaker condition on $N$:

\begin{lemma}
  Let\label{lem:norm_wfpn2} $\ses N\G Q$ be a short exact sequence with
  $Q$ of type $\FP_n$ and $N$ of type $\wFP_{n-1}$. Also, assume
  that $\Hom_0(Q,{H_n(N)})$ is finitely generated. 
  Then $\Hom_k(\G,\Z)$ is finitely generated for $k\leq n$.
\end{lemma}
\begin{proof}
  This follows easily from the LHS spectral sequence
  \[E^2_{pq}=H_p(Q,H_q(N))\abut H_{p+q}(\G).\]
  The assumptions on $N$ and $Q$ imply that
  $E^2_{pq}$ is finitely generated for
  $q\leq n-1$ and $p\leq n$. In particular, $E^2_{pq}$
  is finitely generated whenever $p+q\leq n-1$.
  Since $\Hom_0(Q,{\Hom_n(N)})$
  is finitely generated as well, all the entries
  $E^2_{pq}$ on the diagonal $p+q=n$ are finitely generated.
  Then the same holds true for the corresponding entries
  on the $E^\infty$ page. But for any $k$, the entries $E^\infty_{pq}$ on the diagonal
  $p+q=k$ are the factors
  of a filtration of $\Hom_k(\G)$, and so $\Hom_k(\G)$ is
  finitely generated for all $k\leq n$.
\end{proof}

Conversely, if $\G$ is of type $\wFP_n$, then 
the stated condition for the kernel $N$ holds, whenever
$\Hom_{n+1}(Q,\Z)$ is finitely generated.

\begin{lemma}
  Let\label{lem:norm_wfpn} $\ses N\G Q$ be a short exact sequence of groups.
  Assume that $N$ is of type $\wFP_{n-1}$, that $\G$ is of type $\wFP_n$ and 
  that $Q$ is of type $\FP_n$. 
  
  If $\Hom_{n+1}(Q)$ is finitely generated, then so is
  $\Hom_0(Q,\Hom_n(N))$. If $\G$ is of type $\wFP_{n+1}$ the converse
  holds as well.
\end{lemma}
\begin{proof}
  As before, consider the LHS spectral sequence
  \[E^2_{pq}=\Hom_p(Q,\Hom_q(N))\abut \Hom_{p+q}(\Gamma).\]
  The assumptions on $N$ and $Q$ imply that
  $E^2_{pq}$ is finitely generated for
  $q\leq n-1$ and $p\leq n$. 
  
  Let us first focus
  on the entry at position $(n+1,0)$ in the
  spectral sequence. On the $E^{r+1}$-page,
  $E^{r+1}_{n+1,0}$ is the kernel of the map
  $d_r:E^r_{n+1,0}\to E^r_{n+1-r,r-1}$. The codomain
  of this map is finitely generated for $2\leq r\leq n$
  and thus $E^{r+1}_{n+1,0}=\ker d_r$ is finitely generated
  if and only if $E^r_{n+1,0}$ is. Hence, by induction,
  $E^{n+1}_{n+1,0}$ is finitely generated if and only if
  $E^2_{n+1,0}=\Hom_{n+1}(Q,\Z)$ is.
  
  Similarly, looking at the entry $(0,n)$ we find that
  $E^{r+1}_{0,n}$ is the cokernel of the map
  $d_r:E^r_{r,n-r+1}\to E^r_{0,n}$. The groups
  $E^r_{r,n-r+1}$ are finitely generated for
  $2\leq r\leq n$ and thus
  $E^{r+1}_{0,n}=\coker d_r$ is finitely generated
  if and only if $E^r_{0,n}$ is. By induction,
  $E^{n+1}_{0,n}$ is finitely generated if and only if
  $E^2_{0,n}=\Hom_0(Q,\Hom_n(N))$ is.
  
  On the $E^{n+1}$-page we have a map
  \begin{equation}
    d_{n+1}:E^{n+1}_{n+1,0}\to E^{n+1}_{0,n} \label{eq:the_d}
  \end{equation}
  The cokernel of this map is $E^{n+2}_{0,n}=E^{\infty}_{0,n}$,
  which is finitely generated, since $\G$ is of type $\wFP_n$.
  Thus, if $E^{n+1}_{n+1,0}$ is finitely generated,
  then so is $E^{n+1}_{0,n}$.
  
  If $\G$ is of type $\wFP_{n+1}$, then the kernel of
  \eqref{eq:the_d}, $E^{n+2}_{n+1,0}=E^{\infty}_{n+1,0}$ is
  finitely generated as well.
  So in this case, $E^{n+1}_{n+1,0}$ is finitely
  generated \emph{if and only if} $E^{n+1}_{0,n}$ is.
\end{proof}

Now we can prove the desired ``weak'' version of the
$n$-$(n+1)$-$(n+2)$ Conjecture.

\begin{thm}[The Weak $n$-$(n+1)$-$(n+2)$ Theorem]
  Let\label{thm:weakn12}
  \begin{gather*}
    \ses[\pi_1]{N_1}{\G_1}Q\\
    \ses[\pi_2]{N_2}{\G_2}Q
  \end{gather*}
  be short exact sequences with $N_1$ of type $\wFP_n$,
  $\G_1$ of type $\wFP_{n+1}$, $\G_2$ of type $\FP_{n+1}$,
  and $Q$ of type $\FP_{n+2}$. Then the associated fibre product
  is of type $\wFP_{n+1}$.
\end{thm}
\begin{proof}
  Denote the fibre product by $P$ and the projection homomorphisms
  by $p_1:P\to\G_1$ and $p_2:P\to\G_2$. 
  
  We need to show that $\Hom_k(P')$ is finitely generated for all
  $k\leq n+1$ and all finite index subgroups $P'\leq P$.
  First we will show that it is in fact enough to prove that $H_k(P)$
  is finitely generated for $k\leq n+1$.
  
  Let $P'\leq P$ be a subgroup
  of finite index and set $\G_1':=p_1(P)$ and $\G_2':=p_2(P)$. Then
  $P'$ is itself a subdirect product of $\G_1'$ and $\G_2'$. The associated
  short exact sequences are
  \begin{gather*}
    \ses[\pi_1]{N_1'}{\G_1'}{Q'}\\
    \ses[\pi_2]{N_2'}{\G_2'}{Q'}
  \end{gather*}
  where $N_1':=P'\cap\G_1$, $N_2':=P'\cap\G_2$ and $Q'=P'/(N_1'\times N_2')$.
  Now the $\G_i'$ and $N_i'$ are finite index subgroups in
  $\G_i$ and $N_i$ respectively, so they share the same finiteness properties.  
  Furthermore $Q'\isom\G_1'/N_1'$ is a finite index subgroup
  of $\G_1/N_1'$ which fits in a short exact sequence
  \[\ses{\quot{N_1}{N_1'}}{\quot{\Gamma_1}{N_1'}}{Q}.\]
  Here $N_1/N_1'$ is finite and $Q$ is of type $\FP_{n+2}$, so
  $\G_1/N_1'$ is of type $\FP_{n+2}$ as well.
  
  Thus we have shown that any finite index subgroup
  $P'\leq P$ is itself a fibre product associated to two short exact
  sequences satisfying the same finiteness conditions as
  those in the statement of the theorem. Without loss of generality,
  it therefore suffices to show that $\Hom_k(P,\Z)$ is finitely generated
  for all $k\leq n$.
  
  Consider first the short exact sequence
  \[\ses[p_1]{N_1}{P}{\G_2}.\]
  Here, $\G_2$ is of type $\FP_{n+1}$ and $N_1$
  is of type $\wFP_n$. By Lemma \ref{lem:norm_wfpn2}, if we can
  show that $\Hom_0(\G_2,\Hom_{n+1}(N_1))$ is finitely generated,
  it follows that $\Hom_k(P)$ is finitely generated for $k\leq n$.
  
  Now consider the short exact sequence
  \[\ses[\pi_1]{N_1}{\G_1}Q.\]
  Here, we can apply Lemma \ref{lem:norm_wfpn}
  to conclude that $\Hom_0(Q,{\Hom_{n+1}(N_1)})$ is finitely generated.
  What remains then is to compare $\Hom_0({\G_2},{\Hom_{n+1}(N_1)})$
  with $\Hom_0(Q,{\Hom_{n+1}(N_1)})$, and we will show
  that these two modules are in fact isomorphic, finishing the proof.
  
  Recall that for any group $\G$ and $\G$-module $A$,
  the $0$\textsuperscript{th} homology group $\Hom_0(\G,A)$ is the module of coinvariants
  \[\Hom_0(\G,A)=A_\G=\quot{A}{\sgp{\{a-\g a\with a\in A, \g\in\G\}}}.\]
  Let $\g_2\in\G_2$. The action of $\G_2$ on $H_{n+1}(N_1)$ is induced by the
  conjugation action of $P$ on $N_1$. More precisely,
  given $\g_2\in\G_2$, pick a lift $\g=(\g_1,\g_2)\in P$. Then
  the action of $\g_2$ on $H_{n+1}(N_1)$ is given by
  the map $c^\g_*:H_{n+1}(N_1)\to H_{n+1}(N_1)$
  induced by the conjugation homomorphism
  $c^\g:N_1\to N_1$. Conjugating an
  element of $N_1=N_1\times 1$ by $\g$ has the same effect
  as conjugating by $\g_1$ in $\G_1$:
  \[c^\g(n)=(\g_1,\g_2)(n,1)(\g_1,\g_2)^{-1}=(\g_1n\g_1^{-1},1).\]
  Therefore $c^\g=c^{\g_1}$. Now set $q:=\pi_1(\g_1)=\pi_2(\g_2)$. 
  By definition of the $Q$-action on
  $H_{n+1}(N_1)$ (stemming from the short exact sequence $\ses{N_1}{\G_1}Q$), we have
  \[qa=c^{\g_1}_*(a)=c^\g_*(a)=\g_2 a,\quad\text{for all $a\in H_{n+1}(N_1)$}.\]
  Since any $q\in Q$ lifts to some $\g_2\in\G_2$, we obtain that
\end{proof}

Since the $n$-$(n+1)$-$(n+2)$ Conjecture implies the Virtual Surjections Conjecture,
it should come as no surprise that Theorem \ref{thm:weakn12} implies
a weak variant of the Virtual Surjections Conjecture. The proof
follows closely that of Theorem \ref{thm:implication}.

\begin{cor}[The Weak Virtual Surjections Theorem]
  Let $k\geq 2$. Let $\G_1,\dots,\G_n$ be groups of type $\FP_k$ and $P\leq\G_1\times\dots\times\G_n$
  a subgroup of their direct product. If $P$ virtually surjects
  to $k$-tuples of factors then $P$ is of type $\wFP_k$.\label{cor:weakvsc}
\end{cor}
\begin{proof}
  We will use the notation established at the end of Chapter \ref{sec:fibreproducts}.
  As in Theorem \ref{thm:implication} we prove the following claim by induction on $k$:
  \begin{unnumclaim}
    Let $k\geq 1$. Let $\G_1,\dots,\G_n$ be groups of type $\FP_k$ and $P\leq\G_1\times\dots\times\G_n$
    a subgroup of their direct product such that $P/(P\cap\G_i)$ is virtually
    nilpotent for all $i$ with $1\leq i\leq n$. Then $P$ is of type $\wFP_k$.
  \end{unnumclaim}
  For $k\geq 2$ the extra assumption on the $P/(P\cap\G_i)$ is redundant,
  by \ref{lem:nilp}, so the Claim does imply the statement of the Corollary.
  Furthermore we may assume, by the same argument as used in Section \ref{sec:vsc}
  (see Proposition \ref{prop:vs1}), that $P$ is a subdirect product, replacing
  each $\G_i$ by its finite index subgroup $p_i(P)$, if necessary.
  
  In the base case $k=1$, the groups $\G_1,\dots,\G_n$ are $\FP_1$ and thus finitely generated.
  So from Proposition \ref{prop:vs1} we have that $P$ is finitely generated and thus a fortiori
  of type $\wFP_1$.
  
  Fix a $k\geq 2$. We prove the Claim for this $k$ by induction on $n$.
  If $n\leq k$ then $P$ is of finite index in $\G_1\times\dots\times\G_n$
  and thus of type $\FP_k$. 
  
  So let $n>k$. Set $T:=p_{1,\dots,n-1}(P)\leq\G_1\times\dots\times\G_{n-1}$.
  Then $P$ is the fibre product associated to the short exact sequences
  \begin{alignat*}{2}
  \ses{N_{1,\dots,n-1}}{&\,T&}Q\\
  \ses{N_n}{&\,\G_n\,&}Q
  \end{alignat*}
  Here, $T\leq \G_1\times\dots\times\G_{n-1}$ is a subdirect product, which
  also virtually surjects to $k$-tuples, since $p_J(T)=p_J(P)$
  for all $J\subset\{1,\dots,n-1\}$. By induction, we can assume that
  $T$ is of type $\wFP_k$.
  
  By Lemma \ref{lem:vskernel}, $N_{1,\dots,n-1}\leq\G_1\times\dots\times\G_{n-1}$
  virtually surjects to $(k-1)$-tuples. In addition,
  \[\quot{\G_i}{(N_{1,\dots,n-1}\cap\G_i)}=\quot{\G_i}{P\cap\G_i}\]
  is virtually nilpotent for all $i\in\{1,\dots,n-1\}$. Thus, by the inductive assumption on $k$, 
  we can assume that $N_{1,\dots,n-1}$ is of type 
  $\wFP_{k-1}$.
  
  Furthermore, by assumption, $Q\isom\G_n/N_n$ is virtually nilpotent
  and finitely generated and therefore of type $\Fin_\infty$, so certainly
  of type $\wFP_2$.
  We can then apply Theorem \ref{thm:weakn12} to conclude that $P$ is of type $\wFP_k$, finishing
  the induction.
\end{proof}

Kochloukova in \cite{Kochloukova} proved a converse to this theorem
in the case where the factors are non-abelian free (or limit) groups, which we quoted
earlier (in a weaker form) as Theorem \ref{thm:KochloukovasTheorem}. Her
theorem states \cite[Theorem 11]{Kochloukova}:
\begin{thm}
  Let $\G_1,\dots,\G_n$ be non-abelian limit groups
  and $P\leq\G_1\times\dots\times\G_n$ a full subdirect product.
  If $P$ is of type $\wFP_k$ (for some $k\geq 2$) then $P$ virtually surjects to $k$-tuples.\qed
\end{thm}

So in conjunction with Corollary \ref{cor:weakvsc} this gives a
complete characterisation of the full subdirect products of type $\wFP_k$
of free groups:
\begin{cor}
  Let $\G_1,\dots,\G_n$ be non-abelian limit groups
  and $P\leq\G_1\times\dots\times\G_n$ a full subdirect product.
  Then $P$ is of type $\wFP_k$ (for some $k\geq 2$) if and only if $P$ virtually surjects to $k$-tuples.\qed
  \label{cor:wFPkCharacterisation}
\end{cor}

Of course, the conjecture (which
would follow from the Virtual Surjections Conjecture) is that
``$\wFP_k$'' in this statement can be replaced by ``$\Fin_k$'',
which would imply that these two finiteness properties are equivalent (for $k\geq 2$)
among the class of subgroups of direct products
of free groups.

\section{Abelian quotients}
\label{sec:abelian}

We have found in Section \ref{sec:vsc} that there is particular interest in resolving the
$n$-$(n+1)$-$(n+2)$ Conjecture in the case where the quotient
$Q$ is nilpotent, as that would suffice to establish the
Virtual Surjections Conjecture. One might try then, as a first step towards that goal, 
to prove the conjecture in
the case where the quotient $Q$ is abelian, hoping to
then establish the nilpotent case by induction on the nilpotency class.
We will prove in this chapter the abelian case, followed by an analysis
of the difficulties one comes up against in a potential induction argument.

The case of abelian $Q$ is in fact quite special. Recall the 
standard setup: We are given two short exact
\begin{gather*}
\ses[\pi_1]{N_1}{\Gamma_1}Q\\
\ses[\pi_2]{N_2}{\Gamma_2}Q
\end{gather*}
where $\Gamma_1$ and $\Gamma_2$ are of type $\Fin_{n+1}$,
$N_1$ is of type $\Fin_n$ and $Q$ now is a finitely
generated abelian group (and therefore in particular
of type $\Fin_\infty$). The fibre product
\[P=\{(\gamma_1,\gamma_2)\in\Gamma_1\times\Gamma_2\with\pi_1(\gamma_1)=\pi_2(\gamma_2)\}\]
is then the kernel of a homomorphism to an abelian group, namely
\[\Gamma_1\times\Gamma_2\to Q,(\gamma_1,\gamma_2)\mapsto \pi_1(\gamma_1)-\pi_2(\gamma_2).\]

This innocuous observation simplifies matters enormously: 
While there is a rather short supply of generally applicable 
tools to answer the question which subgroups of a group
of type $\Fin_n$ are themselves of this type, there is
one large exception. That is the case where the subgroup in question
is the kernel of a homomorphism to an abelian group. For here, the
theory of $\Sigma$-invariants can be brought to bear.

\subsection{$\Sigma$-invariants, Meinert's Inequality and the abelian $n$-$(n+1)$-$(n+2)$ Conjecture} 

We content ourselves here with a brief summary of results
from $\Sigma$-theory pertinent to our discussion. The invariants were
first introduced in the context of metabelian group by Bieri and Strebel
in \cite{BieriStrebel} and later generalised in \cite{BNS,BieriRenz,BieriGeoghegan}.
A general reference is \cite{BieriStrebelBook}.
The higher-dimensional topological invariants used here were introduced
in \cite{RenzThesis}. For a survey, see \cite{BieriSurvey,BieriFinitenessLength}.

Consider for any group $\G$ of type $\Fin_n$ the set $\Hom(\G,\R)\setminus\{0\}$
of non-trivial homomorphisms from $\G$ to the additive reals (the ``characters'' of $\G$).
We take two such homomorphisms $\chi_1,\chi_2:\G\to\R$
to be equivalent if they only differ by a positive
real factor, i.e.\ $\chi_1=c\chi_2$ for some $c>0$.
The set of equivalence classes is called the character sphere $S(\G)$ (note that
$\Hom(\G,\R)$ is a finite-dimensional real vector space if $\G$ is finitely
generated, so $S(\G)$ is a topological sphere).

We can then associate to $\G$ a chain of invariants
\[S(\G)\supset\Sigma^1(\Gamma)\supset\Sigma^2(\Gamma)\supset\dots\supset\Sigma^n(\Gamma),\]
the $\Sigma$-\ or Bieri-Neumann-Strebel-Renz invariants of $\Gamma$.

The $\Sigma$-invariants describe certain connectivity properties
of $\E\G$-complexes and are related to a variety of
structural properties of the group, for which we refer the reader to the sources cited above.
For us, the utility of these invariants lies in the fact
that they contain all the information on the higher finiteness
properties of the normal
subgroups of $\Gamma$ with abelian quotient.
The following theorem is due to Renz \cite{RenzThesis,RenzSingapore}
(see \cite{BieriRenz} for a proof of the analogous result for the homological 
invariant or \cite{BieriGeoghegan} for a proof in a much more general setting).

\begin{thm}
  Let $\Gamma$ be a group of type $\Fin_n$ and $N\nsub\Gamma$
  a normal subgroup with $\Gamma/N$ abelian. Then $N$ is
  of type $\Fin_n$ if and only if
  \[S(G,N):=\{[\chi]\in S(G)\with \chi(N)=0\}\]
  is contained in $\Sigma^n$.\qed\label{thm:sigma}
\end{thm}

To apply this theorem to the problem at hand we also need a way of
relating the $\Sigma$-invariants of a direct product to those of the
factor groups. Fortunately, all our needs are served by Meinert's Inequality.
This is due to Meinert (unpublished, based on work in \cite{MeinertDiplom}; for a
proof see \cite[Lemma 9.1]{Gehrke} or the sketch in \cite{BieriFinitenessLength}).

\begin{thm}[Meinert's Inequality]
  Let $\Gamma_1$ and $\Gamma_2$ be groups of type $\Fin_n$ and let
  $\chi:\Gamma_1\times\Gamma_2\to\R$ be a homomorphism
  with both $\chi|_{\Gamma_1}$ and $\chi|_{\G_2}$ non-trivial. If $[\chi|_{\Gamma_1}]\in\Sigma^k(\Gamma_1)$
  and $[\chi|_{\Gamma_2}]\in\Sigma^l(\Gamma_2)$ for some $k,l\geq 1$ with $k+l<n$ then
  $[\chi]\in\Sigma^{k+l+1}(\Gamma_1\times\Gamma_2)$.\qed
\end{thm}

Now we are equipped to give a proof of the $n$-$(n+1)$-$(n+2)$ Conjecture in 
the case where the quotient $Q$ is virtually abelian.

\begin{thm}
  Let
  \begin{gather*}
    \ses[\pi_1]{N_1}{\Gamma_1}Q\\
    \ses[\pi_2]{N_2}{\Gamma_2}Q
  \end{gather*}
  be short exact sequences of groups with $\Gamma_1$ and $\Gamma_2$ of type $\Fin_{n+1}$,
  $Q$ finitely generated virtually abelian, $N_1$ of type $\Fin_k$ and $N_2$ of type $\Fin_l$ for some
  $k,l\geq 0$ with $k+l\geq n$. Then the fibre product associated to these sequences
  is of type $\Fin_{n+1}$.\label{thm:abeliann12}
\end{thm}
\begin{remark}
  Note that the assumptions on the finiteness properties of
  the kernels here are more general than in the original statement, where we asked
  for $N_1$ to be of type $\Fin_n$ but made no assumptions on $N_2$. This asymmetry
  disappears in the abelian case, where we have the more attractive assumption that the finiteness 
  properties of the kernels $N_1$ and $N_2$ ``add up to $n$''. Of course, it is tempting to speculate
  whether this alternative assumption might be sufficient in the general case as well,
  but at this point, the abelian case aside, we have very little evidence to either support
  or reject this.
\end{remark}
\begin{proof}
  First, we note that we may assume that $Q$ is abelian instead of merely
  virtually abelian, by the remark after Theorem \ref{thm:implication}.

  We have to show that $[\chi]\in\Sigma^{n+1}(\G_1\times\G_2)$ for all 
  non-trivial homomorphisms $\chi:\G_1\times\G_2\to\R$ with $P\subset\ker\chi$.
  
  So let $\chi$ be any such a homomorphism. Since we can pick for any
  $\g=(\g_1,\g_2)\in\G_1\times\G_2$ a $\g'=(\g_1',\g_2)\in P$,
  so that $\g^{-1}\g'\in \G_1$, we have
  $\G_1\times\G_2=\G_1P$. So $\chi|_{\G_1}$ cannot be trivial, otherwise
  $\chi$ would be, too. Similarly $\chi|_{\G_2}\neq 0$.
    
  Now $\chi|_{\G_1}$ vanishes on $P\cap\G_1=N_1$. Since $N_1$
  is of type $\Fin_k$, Theorem \ref{thm:sigma} tells us that
  $[\chi|_{\G_1}]\in\Sigma^k(\G_1)$. By an analogous argument,
  $[\chi|_{\G_2}]\in\Sigma^l(\G_2)$. By Meinert's inequality
  this implies $[\chi]\in\Sigma^{k+l+1}(\G_1\times\G_2)\subset\Sigma^{n+1}(\G_1\times\G_2)$.
\end{proof}

Solving the $n$-$(n+1)$-$(n+2)$ Conjecture in this special case means that we can
also establish a corresponding case of the Virtual Surjections Conjecture. This
was previously proved by Kochloukova in \cite{Kochloukova}.

\begin{cor}
  Let $\G_1,\dots,\G_n$ be groups of type $\Fin_k$ and $P\leq\G_1\times\dots\times\G_n$
  a subgroup of their direct product. If $P$ virtually
  surjects to $k$-tuples of factors and $p_i(P)/(P\cap\G_i)$ is virtually
  abelian for all $i$ with $1\leq i\leq n$ then $P$ is of type $\Fin_k$.
  \label{cor:abelianvsc}
\end{cor}
\begin{proof}
  This follows directly from Theorem \ref{thm:implication} (or, more precisely,
  from the Claim in the proof).
\end{proof}

In fact Lemma \ref{lem:nilp} tells us that if
$P\leq\Gamma_1\times\dots\times\Gamma_n$ is a subdirect product that virtually surjects 
to $k$-tuples, where $2k>n$, then the quotients $\Gamma_i/N_i$ are virtually abelian, so we obtain:

\begin{cor}
The Virtual Surjections Conjecture holds if $2k>n$ and in particular for $n\leq 5$.
\label{cor:fewfactorsvsc}
\end{cor}
\begin{proof}
The latter statement follows from the Virtual Surjections to Pairs Theorem, i.e.\ the
$k=2$ case of the Virtual Surjections Conjecture, proved in \cite[Theorem A]{BHMSFinPres}.
\end{proof}

\bibliographystyle{plain}
\bibliography{n12conjecture}

\end{document}